\documentclass[12pt,reqno]{amsart}
\usepackage{fullpage}
\usepackage{amsfonts}
\usepackage{amssymb}
\usepackage{times}
\usepackage{graphicx}
\usepackage{tgpagella}
\usepackage{euler}
\usepackage[T1]{fontenc}
\usepackage{amssymb}
\usepackage{amsmath,amsthm}
\usepackage[colorlinks=true,
            linkcolor=red,
            urlcolor=blue,
            citecolor=blue]{hyperref}
\usepackage{amsthm}
\usepackage[english]{babel}
\usepackage{mathrsfs}
\usepackage{eucal}
\usepackage{calligra}
\usepackage{calrsfs}
\usepackage{enumitem}
\usepackage{mathpazo}

\numberwithin{equation}{section}

\newtheorem{main}{Theorem}

\newtheorem{mcor}[main]{Corollary}

\newtheorem{theorem}{Theorem}[section]
\newtheorem{thrm}[theorem]{Theorem}
\newtheorem{lem}[theorem]{Lemma}
\newtheorem{prop}[theorem]{Proposition}
\newtheorem{conj}[theorem]{Conjecture}
\newtheorem{cor}[theorem]{Corollary}

\theoremstyle{definition}
\newtheorem{defn}[theorem]{Definition}
\newtheorem{notation}[theorem]{Notation}
\newtheorem{Remarks}[theorem]{Remarks}

\newtheorem{claim}[theorem]{Claim}
\newtheorem*{claim*}{Claim}
{ \theoremstyle{remark}\newtheorem*{remark}{Remark} }

\DeclareMathAlphabet{\pazocal}{OMS}{zplm}{m}{n}

\newcommand{\paA}{\pazocal{A}}

\def\ra{\rightarrow}

\def\la{\lambda}
\def\La{\Lambda}
\def\Om{\Omega}

\def\om{\omega}

\def\sg{\sigma}
\def\Sg{\Sigma}

\def\g{\gamma}

\def\G{\Gamma}
\def\Cal{\mathcal}

\newcommand{\Z}{\mathbb{Z}}

\newcommand{\C}{\mathbb{C}}

\newcommand{\Q}{\mathbb{Q}}
\newcommand{\R}{\mathbb{R}}

\theoremstyle{remark}

\newcommand{\set}[1]{ \left\lbrace #1 \right\rbrace }

\begin{document}

\title[Tensor product decompositions of II$_1$ factors arising from extensions of amalgamated free product groups]{Tensor product decompositions of II$_1$ factors arising from extensions of amalgamated free product groups}
\author[I. Chifan]{Ionu\c{t} Chifan}
\address{Department of Mathematics, The University of Iowa, 14 MacLean Hall, Iowa City, IA  
52242, USA}
\email{ionut-chifan@uiowa.edu}
\thanks{I.C. was partially supported by NSF grants DMS \# 1600688 and DMS \# 1301370}

\author[R. de Santiago]{Rolando de Santiago}
\address{Department of Mathematics, University of California Los Angeles, Los Angeles, CA, USA }
\email{rdesantiago@math.ucla.edu}
\thanks{R.dS. was partially supported by  a grant at from Sloan Center for Exemplary Mentoring, NSF grant DMS \# 1600688  and RTG Assistantship at UCLA}

\author[W. Sucpikarnon]{Wanchalerm Sucpikarnon}
\address{Department of Mathematics, The University of Iowa, 14 MacLean Hall, Iowa City, IA  
52242, USA}
\email{wanchalerm-sucpikarnon@uiowa.edu}
\thanks{W.S. was partially supported by NSF grant DMS \# 1600688}

\maketitle





\begin{abstract}  In this paper we introduce a new family of icc groups $\G$ which satisfy the following product rigidity phenomenon, discovered in \cite{DHI16} (see also \cite{dSP17}): \emph{all} tensor product decompositions of the II$_1$  factor $L(\G)$ arise \emph{only} from the canonical direct product decompositions of the underlying group $\G$. Our groups are assembled from certain HNN-extensions and amalgamated free products and include many remarkable groups studied throughout mathematics such as graph product groups, poly-amalgam groups, Burger-Mozes groups, Higman group, various integral two-dimensional Cremona groups, etc. As a consequence we obtain several new examples of groups that give rise to prime factors. 
\end{abstract}

\maketitle


\section{Introduction}
An important step towards understanding the structure of II$_1$ factors is the study their tensor product decompositions.  A factor is called \emph{prime} if it cannot be decomposed as a tensor product of diffuse factors. Using his notion of $*$-orthogonal von Neumann algebras, S. Popa was able to show in \cite{Po83} that the (non-separable) II$_1$ factor $L(\mathbb F_S)$ arising from the free group $\mathbb F_S$ with uncountably many generators $S$ is prime. More than a decade later, using Voiculescu's influential free probability theory, Ge managed to prove the same result about the free group factors $L(\mathbb F_n)$ with countably many generators, $n\geq 2$ \cite{Ge98}.  Using a completely different perspective based on $C^*$-techniques, Ozawa obtained a far-reaching generalization of this by showing that for every icc hyperbolic group $\G$ the corresponding factor $L(\G)$ is in fact \emph{solid} (for every diffuse $A\subseteq L(\G)$ von Neumann subalgebra, its relative commutant $A'\cap L(\G)$ is amenable) \cite{Oz03}.  Developing a new approach rooted in the study of closable derivations, Peterson showed primeness of $L(\G)$, whenever $\G$ is any nonamenable icc group with positive first Betti number \cite{Pe06}. Within the powerful framework of his deformation/rigidity theory Popa discovered a new proof of solidity of the free group factors \cite{Po06}. These methods laid out the foundations of a rich subsequent activity regarding the study of primeness and other structural aspects of II$_1$ factors \cite{Oz04, CH08, CI08, Si10, Fi10,CS11,CSU11, SW11, HV12,Bo12, BHR12, DI12, CKP14, Is14, HI15, Ho15, DHI16, Is16}.      

\subsection{Statements of main results} Drimbe, Hoff and Ioana have discovered in \cite{DHI16} a new classification result regarding the study of tensor product decompositions of II$_1$ factors. Precisely, whenever $\G$ is an icc group that is measure equivalent to a direct product of non-elementary hyperbolic groups then \emph{all} possible tensor product decompositions of the corresponding II$_1$ factor $L(\G)$ can arise \emph{only} from the canonical direct product decompositions of the underlying group $\G$. Pant and the second author showed the same result holds when $\G$ is a poly-hyperbolic group with non-amenable factors in its composition series \cite{dSP17}. In this paper we introduce several new classes of groups for which this tensor product rigidity phenomenon still holds. Our groups arise from natural algebraic constructions involving HNN-extensions or amalgamated free products thus including many remarkable groups intensively studied throughout mathematics such as graph products or poly-amalgams groups.  
\vskip 0.05in 
Basic properties in Bass-Serre theory of groups show that the only way an amalgam $\G_1\ast_\Sg\G_2$ could decompose as a direct product is through its core $\Sigma$. Precisely, if $\G_1\ast_\Sg \G_2 =\La_1\times \La_2$ then there is a permutation $s$ of $\{1,2\}$ so that $\La_{s(1)}\leqslant\Sg$. This further gives $\Sg= \La_{s(1)}\times \Sg_0$, $\G_1= \La_{s(1)} \times \G^0_1$, $\G_2= \La_{s(1)} \times \G^0_2$ for some groups $\Sg_0\leqslant \G^0_1,\G^0_2$  and hence $\La_{s(2)} =\G^0_1\ast_{\Sg_0} \G^0_2$. An interesting question is to investigate situations when this basic group theoretic aspect could be upgraded to the von Neumann algebraic setting. It is known this fails in general since there are examples of product indecomposable icc amalgams whose corresponding factors are McDuff and hence decomposable as tensor products (see Remarks \ref{rk1}). However, under certain indecomposability assumptions on the core algebra (see also Remarks \ref{rk1}) we are able to provide a positive answer to our question.    
\begin{main}\label{tensordecompampf} Let $\G=\G_1\ast_{\Sigma}\G_2$ be an icc group such that $[\G_1:\Sg]\geq 2$ and $[\G_2:\Sg]\geq 3$. Assume that $\Sigma$ is finite-by-icc and any corner of $L(\Sg)$ is virtually prime. Suppose that $L(\G)=M_1\bar\otimes M_2$, for diffuse $M_i$'s. Then there exist decompositions $\Sg=\Omega \times \Sg_0$ with $\Sg_0$ finite, $\G_1= \Omega \times \G^0_1$, $\G_2= \Omega \times \G^0_2$, for some groups $\Sg_0\leqslant \G_1^0,\G^0_2$, and hence $\G =\Omega \times (\G^0_1\ast_{\Sg_0} \G^0_2)$. Moreover, there is a unitary $u\in L(\G)$, $t>0$, and a permutation $s$ of $\{1,2\}$ such that 
\begin{equation}
M_{s(1)}= u L(\Omega)^t u^*\quad{ and }\quad  M_{s(2)} = u L(\G_1^0\ast_{\Sg_0} \G_2^0)^{1/t} u^*.
\end{equation}
\end{main}

The same question can be considered in the realm of non-degenerate HNN-extension groups and a similar approach leads to the following counterpart of Theorem \ref{tensordecompampf}.

\begin{main}\label{tensordecomphnnf} Let $\G={\rm HNN}(\La,\Sigma,{\theta})$ be an icc group such that $\Sg\neq \La \neq \theta(\Sg)$. Assume that $\Sigma$ is finite-by-icc and any corner of $L(\Sg)$ is virtually prime. Suppose that $L(\G)=M_1\bar\otimes M_2$, for diffuse $M_i$'s. Then there exist decompositions $\Sg=\Omega \times \Sg_0$ with $\Sg_0$ finite and $\La= \Omega \times \La_0$. In addition, there is $\omega \in \Om$ such that $\theta= {\rm ad}(\omega)_{|\Om} \times \theta_{|\Sg_0} :\Om \times \Sg_0 \ra \Om \times \La_0$ and hence $\G = \Omega \times {\rm HNN}(\La_0,\Sg_0,\theta_{|\Sg_0})$. Also, there is a unitary $u\in L(\G)$, $t>0$, and a permutation $s$ of $\{1,2\}$ such that  
\begin{equation}
M_{s(1)}= u L(\Omega)^t u^*\quad \text{ and }\quad  M_{s(2)} = u L({\rm HNN}(\La_0,\Sg_0, \theta_{| \Sg_0} ))^{1/t} u^*.
\end{equation}
\end{main}

One way to take this analysis to a next step would be to investigate more complicated groups which can be structurally assembled from amalgamated free products. We denote by $\mathcal{T}$ the class of all icc grous that are either amalgamated free products $\G_1\ast_{\Sigma}\G_2$ satisfying $[\G_1:\Sg]\geq 2$, $[\G_2:\Sg]\geq 3$ and $\Sg\in \mathcal C_{\rm rss}$, or, HNN-extensions ${\rm HNN}(\La,\Sg,\theta)$ satisfying $\Sg\neq \La\neq \theta(\Sg)$ and $\Sg\in \mathcal C_{\rm rss}$. For the precise definition of class $\mathcal C_{\rm rss}$ we refer the reader to \cite[Definition 2.7]{CIK13}, but we point out briefly that it contains all finite groups, all non-elementary hyperbolic groups and all non-amenable free products groups. Then we write $\G\in Quot_n(\mathcal{T})$ and say \emph{$\G$ is an $n$-step extension of groups in $\mathcal{T}$} (or a \emph{poly-amalgam of length $n$}) if $\G$ admits a composition series of length $n$ whose factor groups belong to $\mathcal{T}$ (see Section \ref{sec:FiniteStepExt}). This class covers many important groups such as direct products of polyfree groups, including: central quotients of most surface braid groups; most mapping class groups, Torelli groups, and Johnson kernels of low genus surfaces; etc \cite{CKP14}. Combining our methods with previous deep classification results in von Neumann algebras from \cite{PV11,PV12,Io12,Va13,CIK13,CdSS15} we are able to classify of all tensor decompositions of II$_1$ factors arising from these poly-amalgam groups:

\begin{main}  Let $\mathcal D$ be the class of all groups commensurable to a group in $\mathcal C_{\rm rss}\cup (\text{finite-by-}\mathcal{T})$. Let $\G\in Quot_n(\mathcal{T})\cup (Quot_n(\mathcal{T})$-by-$\mathcal C_{\rm rss})$ be an icc group. Assume that $L(\G)=M_1\bar\otimes M_2$, for diffuse $M_i$'s. Then there exist a decomposition $\G=\G_1\times \G_2$, where $\G_i$ is commensurable to a group in $Quot_{n_i}( \mathcal D)$ for some $n_i\in \overline{1,n}$. Moreover, one can find a unitary $u\in L(\G)$, $t>0$ and a permutation $s$ of $\{1,2\}$ such that 
\begin{equation}
M_{s(1)}=uL(\G_1)^tu^*\quad \text{ and }\quad  M_{s(2)}=uL(\G_2)^{1/t}u^*.
\end{equation}  
 \end{main}

Another remarkable class of groups whose tensor decompositions of the corresponding factors could be completely described are the graph product groups. These are natural generalizations of the right angled Artin and Coxeter groups and were introduced by E. Green in her PhD thesis \cite{Gr90}. The study of graph products and their subgroups has became a trendy subject over the last decade in geometric group theory and many important results have been discovered, \cite{A13,HW08,W11,MO13,AM10}. These groups have been considered in the von Neumann algebraic setting in \cite{CF14} and several structural results (strong solidity, Cartanless) for these algebras have emerged from the works \cite{CF14,Ca16}. Our work investigates new aspects regarding the structure of these algebras.   

To introduce the result we briefly recall the graph product construction. Let $\mathcal G$ be a finite graph with vertex set $\mathcal V$ and without loops or multiple edges. The graph product ${\mathcal G}(\G_v, v\in \mathcal V)$ of a given  family  of \emph{vertex groups} $\{\G_v\}_{v\in \mathcal V}$ is the quotient of the free product $\ast_{v\in \mathcal V} \G_v$ by the relations $[\G_u,\G_v]=1$ whenever $u$ and $v$ are adjacent vertices in $\mathcal G$. For every $v\in \mathcal V$ we denote by ${\rm link}(v)$ the set of all vertices $w\neq v$ that are adjacent to $v$. Given a subset $\mathcal A\subseteq \mathcal V$ we denote by ${\rm link}(\mathcal A)=\cap_{v\in \mathcal A}{\rm link}(v)$. Also, we let $\mathcal V_{\rm star}$ be the collection of all $v\in\mathcal V$ such that ${\rm link}(v)=\mathcal V\setminus \{v\}$. 

It is well known that a graph product group can be decomposed non-canonically as a direct product precisely when the vertices of the underlying graph can be partitioned into two disjoint  sets $\mathcal{V}_1\sqcup \mathcal{V}_2=V $ so that each vertex $v\in \mathcal{V}_1 $ is connected to by an edge to every vertex in $V_2$, and vice versa; equivalently, the complement of the underlying graph has two components.  In  this situation, the factors which decompose $\Gamma $ arise from the sub-graph products corresponding to the subgraphs induced by $\mathcal{V}_1 $ and $\mathcal{V}_2 $.  We show that this basic group factorization phenomenon passes to the von Neumann algebraic structure, thus enabling a complete classification of all tensor product decompositions of $L({\mathcal G}(\G_v, v\in \mathcal V))$, for a large family of graphs $\mathcal G$. Specifically, we have the following 

\begin{main}
Let $\G=\mathcal{G}(\G_v, v\in \pazocal V)$ an icc graph product group. Assume that $ \mathcal V_{\rm star}=\emptyset $ and also $|\G_v|=\infty$ for all $v\in \mathcal V$. Suppose that $L(\G)=M_1\bar\otimes M_2$, for diffuse $M_i$'s. Then one can find a proper subset $ \paA\subset \mathcal V$ such that $\paA\sqcup {\rm link}(\paA)=\mathcal V$, a unitary  $u\in L(\G)$, and a scalar $t>0$ satisfying 
\begin{align*}
M_1 = uL(\G_{\paA})^t u^*\quad\text{ and }\quad M_2  = u L(\G_{{\rm link}(\paA)})^{1/t}u^*.
\end{align*}\end{main}

We note in passing that if $ \mathcal V_{\rm star}\neq\emptyset$ then similar statements could be obtained under the additional assumption that the vertex groups $\G_v$ for $v\in \mathcal V_{\rm star}$ are icc hyperbolic (or biexact \cite{Oz03}).

As an application of the aforementioned classification results we are able to identify several new examples of remarkable groups which give rise to prime factors. In particular these include various well known classes of simple groups---a novelty in the area (see part c) in Corollary \ref{prime}). In addition, we are able to recover some of the results obtained in \cite[Corollary B]{DHI16} (see part d) in Corollary \ref{prime}).    

\begin{mcor}\label{prime} If $\G$ is a group in one of the following classes then $L(\G)$ is prime.
\begin{enumerate}
\item [a)] the integral two-dimensional Cremona group $Aut_k (k[x,y])$, where $k$ is a countable field;
\item [b)] Higman's group $\langle a,b,c,d \, |\, a^b=a^2, b^c=b^2, c^d=c^2, d^a=d^2\rangle$ \cite{Hi51};
\item [c)] Burger-Mozes' groups \cite{BM01}, Camm's groups \cite{Ca51}, Bhattacharjee's groups \cite{Bh94};
\item [d)] $PSL_2(\mathbb Z[S^{-1}])$, where $S$ is a finite collection of primes.
\item [e)] $PGL_2(k[t])$, $PSL_2(k[t])$, where $k$ countable field, e.g. $k $ is a number field.
\item [f)]  $PE_2(R[t])\cong E_2(R[t])/ (I_2)  $ where  $R $ is a finite integral domain or $R$ is a countable commutative integral domain with $|(r)|=\infty $ for every $r\in R\setminus\set{ 0} $.
\end{enumerate} 
\end{mcor}

\subsection{Comments on the proof of Theorem \ref{tensordecompampf}} The proofs of the main results in the paper follow a common general strategy. Although Theorem \ref{tensordecompampf} is not the most technically involved, we believe that a brief outline of its proof will give the reader a solid informal understanding of the underlying method of proof. Assume that $\G=\G_1\ast_\Sg\G_2$ and consider a tensor decomposition of $L(\G)=M_1\bar\otimes M_2$. Making use of the classification of normalizers of subalgebras in amalgamated free products developed in \cite{Io12,Va13} we can assume that a corner of $M_1$ can be intertwined into $L(\Sg)$ inside $L(\G)$, in the sense of Popa, ($M_1 \prec_{L(\G)} L(\Sg)$ \cite{Po03}). This in combination with intertwining techniques from \cite{CKP14} and the virtual primeness assumption on $L(\Sg)$ imply that, up to corners, $M_1$ is actually \emph{spatially commensurable} to $L(\Sg)$ (see Section \ref{sec:commensurablealgebras}). This property, denoted by $M_1 \cong^{com}_{L(\G)} L(\Sg)$, roughly means that a (nontrivial) corner of $M_1$ can by conjugated via a partial isometry from $L(\G)$ onto a finite index subalgebra inside a corner of $L(\Sg)$. Then developing new techniques that build upon some ideas from \cite{CdSS15} (see also \cite{DHI16}) we are able to show that $M_1 \cong^{com}_{L(\G)} L(\Sg)$ always implies that the subgroup generated by $\Sigma$ and its centralizer $C_\G(\Sg)$ has finite index in $\G$. Finally adapting some arguments from \cite{OP03} and \cite{CdSS15} we conclude that $\G$ splits non-trivially as a direct product and the statement follows using basic Bass-Serre group theory considerations and intertwining techniques. 

The aforementioned criterion of deducing the direct product splitting of $\G$ (and hence the tensor decomposition of $L(\G)$)  only from spatial commensurability between a corner of $M_1$ and a corner of a subalgebra $L(\Sigma)$ arising from a subgroup $\Sg<\G$ is the common underlying strategy used in the proofs of the main results of the paper.  

\subsection{Terminology} 
The von Neumann algebras $M$ considered in this paper are endowed with a unital, faithful, normal linear functional $\tau:M\ra \mathbb C$ that is tracial, i.e. $\tau(xy)=\tau(yx)$, for all $x,y\in M$. If $x\in M$ then we denote by $\|x\|$ its operator norm and by $\|x\|_2=\tau(x^*x)^{1/2}$ its $2$-norm. Given a von Neumann algebra $M$, we denote by $\mathcal Z(M)$ its center, by $\mathcal U(M)$ its unitary group, by $\mathcal P(M)$ the set of its projections, and by $(M)_1$ its unit ball with respect to the operator norm. For a set $S\subseteq M$, we denote by $S''$ the smallest von Neumann subalgebra of $M$ which contains $S$.

All inclusions $P\subseteq M$ of von Neumann algebras are assumed unital, unless otherwise specified. 
Given von Neumann subalgebras $P, Q\subseteq M$,  we denote by $E_{P}:M\rightarrow P$ the conditional expectation onto $P$, 
by $P'\cap M=\{x\in M\;|\;[x,P]=0\}$ the relative commutant of $P$ in $ M$,
and by $P\vee Q$ the von Neumann algebra generated by $P$ and $Q$.

Unless otherwise specified, all groups considered in this paper are assumed countable and discrete. Given a group $\G$, we denote by $\{ u_\g\}_{\g\in\G} \subseteq \mathcal U(\ell^2\G)$ its left regular representation given by $u_\g(\xi)(\la) = \xi(\g^{-1}\la)$, for all $\xi\in \ell^2\G$ and $\la\in \G$. The weak operator closure of the linear span of $\{ u_\g \}_{\g\in G}$ in $\mathbb B(\ell^2(\G))$ is called the group von Neumann algebra of $\Gamma$ and is denoted by $L(\G)$. $L(\G)$ is a II$_1$ factor precisely when $\G$ has infinite non-trivial conjugacy classes (icc) \cite{MvN43}.

Let $\G$ be a group. Given subsets $K,F\subseteq \G$, we put $KF=\{ kf | k\in K,f\in F\}$. Given a subgroup $\Sigma<\G$, we denote by $C_{\Sigma}(K)=\{ g \in \Sigma | \g k=k \g, \text{ for all }k\in K\}$  the centralizer of $K$ in $\Sigma$. If $\mathcal C$ and $\mathcal D$ are classes of groups then we say that a group $\G$ is $\mathcal C$-by-$\mathcal D$ if there exist a short exact sequence $1\ra A\ra \G \ra B \ra 1$ with $A \in \mathcal C$ and $B\in \mathcal D$. We denote by $\mathfrak S_n$ the permutation group of $n$ letters. If $m< n $ are positive integers, $\overline{m,n} $ will denote the set $\set{m,\ldots,n} $. 

\section{Intertwining Results in Amalgamated Free Product von Neumann algebras}
\subsection{Popa's intertwining techniques} Over a decade ago, S. Popa has introduced  in \cite [Theorem 2.1 and Corollary 2.3]{Po03} the following powerful analytic criterion to identify intertwiners between arbitrary subalgebras of tracial von Neumann algebras.

In order to study the structural theory of von Neumann algberas S. Popa has introduced the following notion of intertwining subalgebras which has been very instrumental in the recent developments in the classification of von Neumann algebras. Given (not necessarily unital) inclusions $P, Q\subset M$ of von Neumann subalgebras, one says that \emph{a corner of $P$ embeds into $Q$ inside $M$} and writes $P\prec_M Q$ if there exist $ p\in  \mathcal P(P), q\in  \mathcal P (Q)$, a $\ast$-homomorphism $\Theta:p P p\rightarrow qQ q$  and a non-zero partial isometry $v\in qM p$ so that $\Theta(x)v=vx$, for all $x\in pP p$. The partial isometry $v$ is also called an \emph{intertwiner} between $P$ and $Q$. If we moreover have that $P p'\prec_{M}Q$, for any nonzero projection  $p'\in P'\cap 1_P M 1_P$ (equivalently, for any nonzero projection $p'\in\mathcal Z(P'\cap 1_P M 1_P)$), then we write $P\prec_{M}^{s}Q$.  For ease of notation, we write $P\prec Q $ (or $P\prec^s Q $) instead of $P\prec_M Q $ (or $P\prec_M^s Q $) whenever  the ambient algebra $ M$ is understood from  context.

Then in \cite [Theorem 2.1 and Corollary 2.3]{Po03} Popa developed a powerful analytic method to identify intertwiners between arbitrary subalgebras of tracial von Neumann algebras.

\begin{thrm}\cite{Po03} \label{corner} Let $(M,\tau)$ be a separable tracial von Neumann algebra and let $P, Q\subset M$ be (not necessarily unital) von Neumann subalgebras. 
Then the following are equivalent:
\begin{enumerate}
\item $P\prec_M Q$.
\item For any group $\mathcal G\subset  \mathcal U(P)$ so that $\mathcal G''= P$ there is no sequence $(u_n)_n\subset \mathcal G$ satisfying $\|E_{ Q}(xu_ny)\|_2\rightarrow 0$, for all $x,y\in  M$.
\end{enumerate}
\end{thrm} 
For further use we record the following basic intertwining result in amalgamated free products von Neumann algebras.
\begin{prop}\label{intertamalgam} Let $M=M_1\ast_P M_2$ be an amalgamated free product von Neumann algebra.
If for each $i$ there is $u_i\in \mathcal U(M_i)$ so that $E_P(u_i)=0$ then $M\nprec_M M_k$ for all $k=1,2$. 
  \end{prop}
\begin{proof} Let  $u=u_1u_2\in \mathcal U(M)$. Using freeness and basic approximation properties one can see that $\lim_{n\ra \infty}\|E_{M_k} (x u^n y)\|_2=0$ for all $x,y\in M$. Then Theorem \ref{corner} (b) gives the conclusion. 
\end{proof}

\subsection{Finite index inclusions of von Neumann algebras} The {\it Jones index}  of a II$_1$ factors inclusion $P\subseteq M$, denoted  $[M:P]$, is  the dimension of $L^2(M)$ as a left  $P$-module. For various basic properties of finite index inclusions of factors we refer the reader to the groundbreaking work of V.F.R. Jones, \cite{Jo81}. Subsequently, several generalizations of the finite index notion for an inclusion of arbitrary von Neumann algebras emerged. For instance, M. Pimsner and S. Popa  discovered a ``probabilistic'' notion of index for an inclusion $P\subseteq M$ of arbitrary von Neumann algebras, which  the case of inclusions of II$_1$ factors coincides with Jones' index, \cite[Theorem 2.2]{PP86}. Namely, the {\it probabilistic index} of a tracial von Neumann algebras inclusion $P\subseteq M$ is defined as $[M:P]=\lambda^{-1}$, where $$\lambda=\inf\;\{ \|E_P(x)\|_2^2 \|x\|_2^{-2} \;\vert \; x\in M_{+}, \; x\not=0\}.$$
Here we convene that $\frac{1}{0}=\infty$.

Below we record some basic properties of finite index inclusions of von Neumann algebras that will be needed throughout the paper. They are well known and some of the proofs are included here just for the sake of completeness.

\begin{prop}\label{finiteindexbasicprop} Let $N\subseteq M$ be von Neumann algebras with $[M:N]<\infty$. Then the following hold:\begin{enumerate}
\item \label{20} If $N$ is a factor then $dim_{\mathbb C}(N'\cap M)\leq [M:N]+1$. 
\item \label{20'} If $\mathcal Z(M)$ is purely atomic then $\mathcal Z(N)$ is also purely atomic, \cite[1.1.2(iv)]{Po95}.  
\item\label{20''} If $N$ is a factor and $r\in N'\cap M$ then $[rMr:Nr]\leq \tau(r) [M:N]<\infty$, \cite[1.1.2(ii)]{Po95}. 
\item \label{20'''} Assume $P\subseteq N\subseteq M$ are II$_1$ factors and there exists $0\not=e\in\mathcal P(P'\cap M)$ satisfying $P e=eM e$ and $E_{N}(e)=\tau(e) 1$. Then $[N:P]<\infty$, and there is $f\in \mathcal P (N'\cap M)$ so that $f\geq e$ and  $[fM f:N f]<\infty$.
\end{enumerate}
\end{prop}

\begin{proof} (\ref{20}) Fix $0\neq p\in N'\cap M$ a projection. Since $N$ is a factor then $E_N(p)=\tau(p)1$. As $[M:N]<\infty$  then $\tau(p)^2=\|E_N(p)\|_2^2\geq [M:N]^{-1}\|p\|_2^2=[M:N]^{-1}\tau(p)$. Thus $\tau(p)\geq [M:N]^{-1}$ for all projections $p\in N'\cap M$ and hence $dim_{\mathbb C}(N'\cap M)\leq [M:N]+1$. 

(\ref{20'}) Let $p\in \mathcal Z(N)$ be a maximal projection such that $\mathcal Z(N)p$ is purely atomic and $\mathcal Z(N)(1-p)$ is diffuse. To prove the conclusion it suffices to show that $q=1-p$ vanishes. Since 
$N\subseteq M$ is finite index so is $qNq\subseteq qMq$. This implies that $qMq\prec_{qMq} qNq$ and hence $qNq'\cap qMq\prec_{qMq} qMq'\cap qMq=\mathcal Z(M)q$. Therefore $\mathcal Z(N)q\prec \mathcal Z(M)q$ and since $\mathcal Z(M)$ is purely atomic it follows that there exists a minimal projection of $\mathcal Z(N)$ under $q$. This forces $q=0$, as desired.

(\ref{20''}) Since $r\in N'\cap M$ and $N$ is a factor we have $E_N(r)=\tau(r)1$. Thus $E_{Nr}(rxr)= \tau(r)^{-1} E_{N}(rxr)  r$ for all $x\in M$. Hence $\|E_{Nr}(rxr)\|^2_{2,r}=\tau(r)^{-2} \|E_N(rxr)\|_2^2\geq \tau(p)^{-2} [M:N]^{-1}\|rxr\|^2_2=\tau(r)^{-1} [M:N]^{-1}\|rxr\|^2_{2,r}$ which shows $[rMr:Nr]\leq \tau(r)[M:N]$.  

(\ref{20'''}) First we show that for all $x\in N$ we have \begin{equation}\label{expect}E_{P}(x)e=exe.
\end{equation}
If $y\in P$, then $\tau(E_{P}(x)ey)=\tau(E_{P}(x)E_{P}(e)y)=\tau(e)\tau(E_{P}(x)y)=\tau(e)\tau(xy)$, and similarly 
 $\tau(exey)=\tau(xye)=\tau(xyE_{N}(e))=\tau(e)\tau(xy)$. Since $E_{P}(x)e, exe\in P e$,   \eqref{expect} follows.
 
 \vskip 0.03in
 \noindent
 Let $f$ be the central support of $e$ in $\{N, e\}''$. Using  \eqref{expect} we get that $\{N f,e\}''$ is canonically isomorphic to the basic construction of the inclusion $P f\subseteq N f$.

 \vskip 0.03in
 \noindent
Now, since $P e=eM e$ and $M$ is a factor, we get that $e(P'\cap M)e=\mathbb Ce$. In particular, $\mathcal Z(\{N,e\}'')e=\mathbb Ce$, and hence $\mathcal Z(\{N f,e\}'')=\mathcal Z(\{N,e\}'')f=\mathbb Cz$. Thus, $\{N f,e\}''$ is a II$_1$ factor.
Applying  \cite[Proposition 3.17]{Jo81}, the inclusion $P f\subseteq N f$ has finite index. Since $E_N(f)=\tau(f)1$ then $P\subseteq N$ has finite index as well.

\vskip 0.03in
\noindent
 Finally, since $\{N f,e\}''$ is a factor and $e\{N f,e\}''e=eM e$, it follows that $\{N f,e\}''=fM f$. Applying \cite[Proposition 3.17]{Jo81} again we get $[fM f:N f]=[N f:P f]<\infty.$
\end{proof}

\begin{lem}\label{finiteindeximage}Let $N\subseteq M$ be a finite index inclusion of II$_1$ factors. Then one can find projections $p\in M$, $q\in N$, a partial isometry $v\in M$, and a unital injective $*$-homomorphism $\phi:pMp \ra qNq$ such that\begin{enumerate} 
\item $\phi(x)v=vx \text{ for all } x\in pMp$, and
\item $[qNq: \phi(pMp)]<\infty$.
\end{enumerate}
\end{lem}
\begin{proof} Since $ [M:N]<\infty$ then $M\prec N$. Thus there exist projections $p\in M$, $q\in N$, a partial isometry $v\in M$, and a unital injective $*$-homomorphism $\phi:pMp \ra qNq$ so that 
\begin{equation}\label{11}
\phi(x)v=vx \text{ for all } x\in pMp.
\end{equation}
Denoting by $Q=\phi(pMp)\subseteq qNq$ notice that $vv^*\in Q'\cap qMq$ and $v^*v=p$. Moreover by restricting $vv^*$ if necessary we can assume wlog the support projection of $E_N (vv^*)$ equals $q$. Also from (\ref{11}) we have that $Qvv^*=vMv^*=vv^*Mvv^*$. Since $M$ is a factor, passing to relative commutants we have $vv^*(Q'\cap qMq)vv^*=(Qvv^*)'\cap vv^*Mvv^*=\mathcal Z(vv^*Mvv^*)= \mathbb C vv^*$. Since $Q'\cap qNq\subseteq Q'\cap qMq$ there is a projection $r\in Q'\cap qN q$ such that $r (Q'\cap q N q) r=  Q r'\cap r N r=\mathbb C r$. Since $q= s(E_N(vv^*))$ one can check that $rv\neq 0$. Thus replacing $ Q$ by $Qr$, $\phi(\cdot)$  by $\phi(\cdot)r$, $q$ by $r$, and $v$ by the partial isometry from the polar decomposition of $rv$ then the intertwining relation (\ref{11}) still holds with the additional assumption that $Q'\cap qM q=\mathbb Cq$. In particular, $E_{q N q}(vv^*)=c q$ where $c$ is a positive scalar.  

To finish the proof we only need to argue that $[qNq:Q]<\infty$. Consider the von Neumann algebra $\langle qNq, vv^*\rangle$ generated by $qNq$ and $vv^*$ inside $qMq$. Therefore we have the following inclusions  $Q\subseteq qNq\subseteq \langle qNq, vv^*\rangle\subseteq qMq$. Since $vv^* Mvv^* =Qvv^*$ then $vv^*  qNq, vv^*= Qvv^* $. Moreover since $vv^*\in Q'\cap qMq$ and $E_{q N q}(vv^*)=c1$ one can check that $\langle qNq, vv^*\rangle$ is isomorphic to the basic construction of $Q\subseteq qNq$. Therefore $Q\subseteq qNq$ has index $c$ (hence finite).  \end{proof}

\subsection{Relative amenability for von Neumann algebras}
 Let $P\subset M$ be an inclusion of von Neumann algebras. A state $\phi:M\to \C $ is \emph{$P$-central} if $\phi(mx)=\phi(xm) $ for every $x\in P $ and every $m\in M $.  A tracial von Neuman algebra $(M,\tau) $ is  \emph{amenable} if there exists an $M$-central state on $\phi:B(L^2(M,\tau))\to \C $ with $\phi|_M=\tau$.  By the celebrated result of A.~Connes, a von Neumann algebra is amenable if and only if it is approximately finite dimensional, i.e.~$M=\varinjlim M_n $ for an increasing sequence of finite dimensional algebras $M_n$, \cite{Co76}.  In \cite{OP07} it was considered a relative version of this notion which turned out very instrumental in studying structural properties of von Neumann algebras.

\begin{defn} \cite[Section 2.2]{OP07}
Let $(M,\tau)$ be a tracial von Neumann algebra, $p\in M$ a projection, and $P\subset pMp,Q\subset M$ von Neumann subalgebras. We say that $P$ is {\it amenable relative to $Q$ inside $M$} if there exists a $P$-central state $\phi:p\langle M,e_Q\rangle p\rightarrow\mathbb C$ such that $\phi(x)=\tau(x)$, for all $x\in pMp$.
\end{defn}
One easily  recovers classical notion of amenability  by choosing $ Q=\mathbb{C}$; more generally, if $Q$ is amenable and $P $ is amenable relative to $Q$,  it necessarily follows that $P$ is amenable. This definition is fruitful for numerous reasons, not the least of which its strong parallel properties. Also we will often use that every amenable $P$ is automatically amenable with respect to any subalgebra $Q$ inside $M$.  

As an amenable group Neumann algebra corresponds to an amenable group, so does relative amenability of  group von Neumann algebras coincide with relative amenability of the groups.  More specifically, given a group $\Gamma $ and a pair of subgroups  $\Lambda_1,\Lambda_2\leqslant \Gamma $,   $\Lambda_1<\Gamma $ is amenable relative to $\Lambda_2 $ inside $\Gamma$ i if and only if $L(\Lambda_1 )$ is amenable relative to $L(\Lambda_2)$ inside $L(\Gamma) $.


\section{Product Decompositions of AFP von Neumann Algebras}

In this section we obtain several structural results describing all possible direct product decompositions for large classes of II$_1$ factors that arise either as AFP or HNN von Neumann algebras (see Theorems \ref{tensordecompafp}, \ref{tensordecomphnn}). Along the way we also prove a few intertwining results for commuting subalgebras inside AFP and HNN extensions that are essential for the next sections. Our arguments rely heavily on the classification results of normalizers of subalgebras in AFP and HNN von Neumann algebras from \cite{Io12,Va13}.

\begin{thrm}\label{intertwiningincore1} Let $P\subset M_i$ for $i=1,2$ be von Neumann algebras such that for each $i=1,2$ there is $u_i\in \mathcal U(M_i)$ so that $E_P(u_i)=0$. Consider $M=M_1\ast_P M_2$ be an amalgamated free product von Neumann algebra and assume in addition that $M$ is not amenable relative to $P$ inside $M$.  Let $p\in M$ be a nonzero projection and assume $A_1,A_2 \subseteq pMp$ are two commuting diffuse subalgebras that $A_1\vee A_2\subseteq pMp$ has finite index. Then $A_i\prec P$ for some $i=1,2.$
\end{thrm}

\begin{proof} Fix $A\subset A_1$ arbitrary diffuse amenable subalgebra of $A_1$. Using \cite[Theorem A]{Va13}, one of the following holds:
\begin{enumerate}
\item \label{e1} $A\prec P$;
\item \label{e2} $A_2\prec M_i$ for some $i=1,2$; or 
\item \label{e3} $A_2$ is amenable relative to $P$ inside $M$.
\end{enumerate}
\noindent If (\ref{e2}) holds then either 
\begin{enumerate}[resume]
\item \label{e4} $A_2\prec P$; or 
\item \label{e5} $A_1\vee A_2\prec M_i$. 
\end{enumerate}
\noindent If (\ref{e5}) holds, since 
$[pMp:A_1\vee A_2]<\infty$, then we must have $M\prec_M M_i$. Then Proposition \ref{intertamalgam} will lead to a contradiction. If case (\ref{e3}) holds, then applying \cite[Theorem A]{Va13} again 
we get one of the following
\begin{enumerate}[resume]
\item \label{e6} $A_2\prec P$;
\item \label{e7} $A_1\vee A_2$ is a amenable relative to $P$ inside $M$; or 
\item \label{e8}$A_1\vee A_2\prec M_i$ for some $i.$
\end{enumerate}

\noindent If (\ref{e7}) holds, since $[pMp:A_1\vee A_2]<\infty$, it follows that $pMp$ is a amenable relative to $P$ inside $M$, contradicting the initial assumption. Notice (\ref{e8}) was already eliminated before.
Summarizing, we have obtained that for any subalgebra $A\subset A_1$ amenable we have either $A\prec P$ or $A_2\prec P.$ Using \cite[Appendix]{BO08} this implies either $A_1\prec P$ or $A_2\prec P.$\end{proof}

Since HNN extensions are corners of AFP \cite{Ue08}, then the previous result implies the following counterpart for HNN-extensions. We leave the details to the reader.

\begin{thrm}\label{intertwiningincore2} Let $P\subset N$ be finite von Neumann algebras and let $\theta: P\ra N$ be a trace preserving embedding.   
Consider $M=\text{HNN}(N,P,\theta)$ be the corresponding HNN-extension von Neumann algebra and assume $M$ is not amenable relative to $P$. Let $p\in M$ be a nonzero projection and assume $A_1,A_2 \subseteq pMp$ are two commuting diffuse subalgebras such that $A_1\vee A_2\subseteq pMp$ is finite index. Then $A_i\prec P$ for some $i=1,2.$
\end{thrm}

\begin{cor}\label{intertwiningcoregroupsafphnn} Assume one of the following holds: \begin{enumerate}\item $\G=\G_1\ast_\Sigma \G_2$ such that $[\G_1:\Sigma]\geq 2$ and $[\G_2:\Sigma]\geq 3$; 
\item $\G={\rm HNN}(\G,\Sigma, \theta)$ such that $\Sigma\neq \G\neq \theta(\Sigma)$.
\end{enumerate}

Denote by $M=L(\G)$ let $p\in \mathcal P(M)$ and assume $A_1,A_2 \subseteq pMp$ are two commuting diffuse subalgebras such that $A_1\vee A_2\subseteq pMp$ is finite index. Then $A_i\prec_M L(\Sigma)$ for some $i=1,2.$
\end{cor}

\begin{proof} Since $[\G_1:\Sigma]\geq 2$ and $[\G_2:\Sigma]\geq 3$ then by the proof of \cite[Theorem 7.1]{Io12} it follows that $L(\G)$ is not amenable relative to $L(\Sg)$. The conclusion follows then from Theorem \ref{intertwiningincore1}. The HNN extension case follows similarly.\end{proof}

\begin{thrm}\label{tensordecompafp} Let $M=M_1\ast_P M_2$ be an amalgamated free product such that $M, M_1, M_2, P$ are II$_1$ factors and $[M_k:P]=\infty$ for all $k=1,2$. Also assume $A_1,A_2 \subset M$ are diffuse factors such that $M=A_1\bar\otimes A_2$. Then there exist 
tensor product decompositions $P=C\bar \otimes P_0$ , $M_1 =C\bar\otimes M^0_1$, $M_2 =C\bar\otimes M^0_2$ and hence $M= C\bar\otimes (M^0_1\ast_{P_0} M^0_2)$. Moreover, there exist $t>0$ and a permutation $\sigma \in \mathfrak S_2$ such that $A_{\sigma(1)}^t\cong  C$ and $A_{\sigma(2)}^{1/t}\cong M^0_1\ast_{P_0} M^0_2$. 
\end{thrm}

\begin{proof} By Theorem \ref{intertwiningincore1} we have that 
$A_1\prec P$ and hence there exist projections $0\neq a\in A_1$, $ 0\neq p\in P$ a partial isometry $0\neq v\in M$ and a unital injective $\ast$-homomorphism $\Phi: aA_1 a\ra pPp$ such that 
\begin{equation}\label{-1}
\Phi(x)v=vx \text{ for all } x\in aA_1 a.
\end{equation}    

Shrinking $a$ if necessary we can assume there is an integer $m$ such that $\tau(p)=m^{-1}$. Letting $B=\phi(aA_1a)$ we have that $vv^*\in B'\cap pMp$. Also we can assume wlog that  $s(E_P(vv^*))=p$ and using factoriality of $A_i$ that  $v^*v= r_1\otimes r_2$. Thus by (\ref{-1}) there is a unitary $u\in M$ so that 
\begin{equation}\label{-3}Bvv^*=vA_1v^*= u r_1 A_1 r_1 \otimes r_2 u^*.
\end{equation}

Passing to relative commutants we also have \begin{equation}\label{-4}\begin{split}vv^* (B'\cap pMp )vv^* &= (Bvv^*)'\cap vv^* Mvv^*=(vA_1v^*)'\cap vMv^* \\& = u ((r_1 A_1 r_1 \otimes r_2 )'\cap r_1 A_1 r_2 \otimes r_2 A_2 r_2 )u^* = u r_1 \otimes r_2A_2 r_2u^*.\end{split}
\end{equation} 

Altogether, we have $vv^*(B\vee B'\cap pMp)vv^*= u (r_1 A_1r_1) \bar\otimes (r_2A_2 r_2 )u^* =vv^* Mvv^*$. Letting $z$ be the central support of $vv^*$ in $B\vee B'\cap pMp$ we conclude that \begin{equation}\label{-2}(B\vee B'\cap pMp)z=zMz.
 \end{equation}
Note by construction we actually have $z\in \mathcal Z(B'\cap pMp)$. In addition, we have $p\geq z\geq vv^*$ and hence $p\geq  s(E_P(z))\geq s(E_P(vv^*))=p$. Thus  $ s(E_P(z))=p$. Also notice that $p\geq s(E_{M_k}(z))\geq z\geq vv^*$.  For every $t>0$ denote by $e^k_t=\chi_{[t,\infty )( E_{M_k}(z))}$.  Using relation (\ref{-2}) and \cite[Lemma 2.3]{CIK13} it follows that the inclusion $(B\vee B'\cap pM_k p)e^k_t\subseteq e^k_t M_k e^k_t$ is finite index. This, together with the assumptions and \cite[Lemma 3.7]{Va07} further imply that $(B\vee B'\cap pM_k p)e^k_t\nprec_{M_k} P$. But $e^k_t z$ commutes with $(B\vee B'\cap pM_k p)e^k_t$ and hence by \cite[Theorem 1.2.1]{IPP05} we have $e^k_tz\in M_k$. Since $e^k_t z\ra z$ in $WOT$, as $t\ra 0$, we obtain that $z\in pM_kp$, for all $k=1,2$. In conclusion $z\in pM_1p\cap pM_2p=pPp$ and hence $z=p$. Thus using factoriality and (\ref{-2}) we get that $pMp= B \bar \otimes (B'\cap pMp)$. Moreover,  we have $B\subset pPp\subset pMp= B \bar \otimes (B'\cap pMp)$ and since $B$ is a factor it follows from \cite[Theorem A]{Ge96} that $pPp= B\bar\otimes (B'\cap pPp)$. Similarly one can show that $pM_kp= B\bar\otimes (B'\cap pM_kp)$ for all $k=1,2$. Thus, 
$$B'\cap pMp = (B'\cap pM_1p) \vee (B'\cap pM_2p)= (B'\cap pM_1p)\ast_{(B'\cap pPp)} ( B'\cap pM_2p).$$
Combining these observations, we now have
\begin{equation*}\begin{split}
pMp= B\bar\otimes (B'\cap pMp)& = B\bar \otimes ((B'\cap pM_1p)\ast_{(B'\cap pPp)} ( B'\cap pM_2p)) \\&=(B\bar\otimes (B'\cap pM_1p))\ast_{B\bar \otimes (B'\cap pPp)} (B\bar \otimes ( B'\cap pM_2p))\end{split} 
\end{equation*} 

Tensoring by $M_m(\mathbb C)$ this further gives \begin{equation*}\begin{split}
M&= M_m(\mathbb C)\bar \otimes pMp\\
&= M_m(\mathbb C)\bar \otimes B\bar \otimes ((B'\cap pM_1p)\ast_{(B'\cap pPp)} ( B'\cap pM_2p)) \\&=(M_m(\mathbb C)\bar \otimes B\bar\otimes (B'\cap pM_1p))\ast_{M_m(\mathbb C)\bar \otimes B\bar \otimes (B'\cap pPp)} (M_m(\mathbb C)\bar \otimes B\bar \otimes ( B'\cap pM_2p))\end{split} 
\end{equation*} 

Letting $C:= M_m(\mathbb C)\bar \otimes B$, $P_0:= B'\cap pPp$, $M^0_k:=B'\cap pM_kp$, altogether, the previous relations show that $P=C\bar \otimes P_0$, $M_1 =C\bar\otimes M^0_1$, $M_2 =C\bar\otimes M^0_2$, and $M= C\bar\otimes (M^0_1\ast_{P_0} M^0_2)$.

For the remaining part of the conclusion, notice that relations (\ref{-3}), (\ref{-4}) and $p=z(vv^*)$ show that $A_i^{\tau(r_1)} \cong B$, $A_{i+1}^{\tau(r_2)}\cong (B'\cap pMp)^{\tau(vv^*)}$. Using amplifications these further imply that $A_i^{m \tau(r_1)} \cong C$, $A_{i+1}^{\tau(r_2)/ (m\tau(vv^*))}\cong M^0_1\ast_{P_0}M^0_2$. Letting $t= m\tau(r_1)$ we get the desired conclusion.\end{proof}



\section{Commensurable von Neumann Algebras}\label{sec:commensurablealgebras}

In the context of Popa's concept of weak intertwining of von Neumann algebras we introduce a notion of commensurable von Neumann algebras up to corners. This notion is essential to this work as it can be used very effectively to detect tensor product decompositions of II$_1$ factors (see Theorems \ref{fromrelcomtocomgroups} and  \ref{virtualprod} below). In the first part of section we build the necessary technical tools to prove these two results. Several of the arguments developed here are inspired by ideas from \cite{CdSS15} and \cite{DHI16}.  

\begin{defn} Let $P,Q\subset M$ (not necessarily unital) inclusions of von Neumann algebras. We write $P\cong^{com}_M Q$ (and we say \emph{a corner of $P$ is spatially commensurable to a corner of $Q$}) if there exist nonzero projections $p\in P$, $q\in Q$, a nonzero partial isometry $v\in M$ and a $\ast$-homomorphism $\phi: pPp\ra qQq $ such that \begin{eqnarray}
 \label{com1}& \phi(x)v=vx \text{ for all }x\in pPp\\
 \label{com2}&[qQq: \phi(pPp) ]<\infty\\
 \label{com3}& s(E_Q(vv^*))=q.
\end{eqnarray} 

\noindent  When just condition (\ref{com1}) is satisfies together with  $\phi(pPp)=qQq$ (i.e $\phi$ is a $\ast$-isomorphism ) we write $pPp\cong^{\phi,v}_M qQq$.
\end{defn}
\begin{remark} When $pPp$ is a II$_1$ factor then so is $\phi(pPp)$. By Proposition \ref{finiteindexbasicprop} (1), $\mathcal Z(qQq)$ is finite dimensional, so there exists $r\in \mathcal Z(qQq)$ such that $rv\neq 0$. Thus replacing $\phi(\cdot)$ by $\phi(\cdot)r$ and $v$ by the isometry in the polar decomposition of $rv$ one can check  (\ref{com1}) still holds. Also from Proposition \ref{finiteindexbasicprop} (3) it follows that $\phi(pPp)r\subseteq rQr$ is an finite index inclusion of II$_1$ factors. Hence throughout this article, whenever $P\cong^{com}_M Q$ and $P$ is a factor, we will always assume the algebras in (\ref{com2}) are II$_1$ factors.  
\end{remark}

We record next a technical variation of \cite[Proposition 2.4]{CKP14} in the context of commensurable von Neumann algebras that will be essential to deriving the main results of this section.

\begin{lem}\label{intertwiningdichotomy} Let $\Sg<\G$ be groups where $\G$ is icc. Assume $\mathcal Z(L(\Sigma))$ is purely atomic, $r\in L(\G)$ is a projection, and there exist commuting II$_1$ subfactors $P,Q\subseteq rL(\G)r$  such that $P\vee Q\subseteq rL(\G)r$ is finite index. If $P\prec L(\Sigma)$ then one of the following holds:
\begin{enumerate} 
\item\label{p1}  There exist projections $p\in P, e\in L(\Sg) $, a partial isometry $w\in M$,  and a unital injective $*$-homomorphism $\Phi: pPp\ra eL(\Sg)e$ such that \begin{enumerate}
\item\label{11'''''} $\Phi(x)w=wx \text{ for all } x\in pPp$;
\item\label{11''''''} $s(E_{L(\Sg)}(ww^*))=e$;
\item \label{12'''''} If $B:=\Phi(pPp)$ then $B\vee (B'\cap eL(\Sg)e) \subseteq eL(\Sg)e$ is a finite index inclusion of II$_1$ factors.
 \end{enumerate}
\item\label{p2} $P\cong^{com}_M L(\Sigma)$.
\end{enumerate}
\end{lem}

\begin{proof} Let $M=L(\G)$. Since $P\prec_M L(\Sg)$ there exist projections $p\in P, q\in L(\Sg) $, a partial isometry $v\in M$,  and a unital injective $*$-homomorphism $\phi: pPp\ra qL(\Sg)q$ such that 
\begin{align}\label{11'''''a}
\phi(x)v=vx \text{ for all } x\in pPp.
\end{align}
Let $C:=\phi(pPp)$. Note $v^*v\in pPp'\cap pMp$, $vv^*\in C'\cap qMq$ and we can also assume that \begin{equation}
s(E_{L(\Sg)}(vv^*))=q
\end{equation} Since $P\vee (P'\cap rMr) $ has finite index in $rMr$ then \cite[Proposition 2.4]{CKP14} implies that
\begin{align}\label{12'''''b}
C\vee (C'\cap qL(\Sg)q) \subset qL(\Sg)q\end{align} is also a finite index inclusion of algebras. 

By Proposition \ref{finiteindexbasicprop}(2)  $\mathcal Z(C'\cap qL(\Sg)q)$ is purely atomic and there is $e\in \mathcal Z(C'\cap qL(\Sg)q)$ so that $ev\neq 0$ and we have either 
\begin{enumerate}[label=\roman*)]
\item
$ (C'\cap qL(\Sg)q)e$ is a II$_1$ factor, or \label{13a}
\item$(C'\cap qL(\Sg)q)e={\rm M}_n(\mathbb C) e$ for some $n\in \mathbb N$.\label{13b}\end{enumerate} 
Consider $\Phi: pPp\ra Ce=:B$ given by $\Phi(x)=\phi(x)e$ for all $x\in pPp$ and let $w$ be the partial isometry in the polar decomposition of $ev $.  Then (\ref{11'''''a}) implies that $\Phi(x)w=wx$ for all $x\in pPp$. Moreover we have $evv^*e\leq ww^*$ and hence $E_{L(\Sg) }(vv^*)e= E_{L(\Sg) }(evv^*e)\leq E_{L(\Sg)}(ww^*)$. Thus  $e=s(E_{L(\Sg) }(vv^*))e=s(E_{L(\Sg) }(vv^*)e)\leq s(E_{L(\Sg)}(ww^*))$ and hence $s(E_{L(\Sg)}(ww^*))=e$. 

Assume case \ref{13a} above. Using (\ref{12'''''b}),  $B\vee (B'\cap eL(\Sg)e)= C e\vee (C'\cap qL(\Sg)q)e\subseteq eL(\Sg)e)$ is a finite index inclusion of II$_1$ factors. Altogether, these lead to possibility (\ref{p1}) in the statement.

Assume case \ref{13b} above.  Then relation (\ref{12'''''b}) implies that  $C=B e\subseteq  eL(\Sg)e$ is finite index which gives possibility (\ref{p2}) in the statement. \end{proof}

\begin{prop}[Lemma 1.4.5 in \cite{IPP05}]\label{transitivity} Let $R,Q\subseteq N \subseteq  M$, $P\subseteq M$ be (not necessarily unital) inclusions of von Neumann algebras such that $1_Q=1_N$ and assume $P$ is factor. Assume that $P\cong^{\phi,v}_M Q$ and $r\in Q'\cap N$ is a projection such that $s(E_N(vv^*))=1_N$ and $rv\neq 0$.  If $Qr\cong^{com}_N R$ then $P\cong^{com}_M R$.
\end{prop}

To be able to derive the first main result of the section we need to recall two results. The first provides a von Neumann algebraic criterion deciding when is a group commensurable to product of infinite groups. This technical tool was an essential piece in the work \cite{CdSS15} and has also lead for other important subsequent developments \cite{CdSS15}. For its proof we refer the reader to the proofs of Claims 4.7-4.12 in \cite{CdSS15}.

\begin{thrm}[Claims 4.7-4.12 in \cite{CdSS15}]\label{fromrelcomtocomgroups} Let $\Sigma<\Lambda$ be finite-by-icc groups. Also assume there exists $0\neq p\in \mathcal Z(L(\Sigma)'\cap L(\La))$ such that $L(\Sigma)\vee L(\Sigma)'\cap L(\Lambda)p\subseteq pL(\Lambda)p$ admits a finite Pimnser-Popa basis. Then there exists $\Omega<\Lambda$ such that $[\Sigma,\Omega]=1$ and $[\Lambda:\Sigma\Omega]<\infty$.  
\end{thrm}

The second result is a basic von Neumann's projections equivalence property for inclusions of von Neumann algebras. Its proof is standard and we include it only for reader's convenience. 

\begin{lem}\label{intinsubfactor} Let $N\subseteq (M, \tau)$ be finite von Neumann algebras, where $N$ is a II$_1$ factor. Then for every projection $0\neq e\in M$ there exists a projection $f\in N$ and a partial isometry $w\in M$ such that $e=w^*w$ and $ww^*=f$. 
\end{lem}

\begin{proof} First we show for every $0\neq y\in \mathcal P(N)$, its central support in $M$ satisfies $z_M(y)=1$. To see this note $z_M(y)\geq y$ and hence $E_N(z_M(y))\geq E_N(y)=y$. As $N$ is a factor then    $E_N(z_M(y))=\tau(z_M(y))1$ and hence $\tau(z_M(y)) 1\geq y$. This forces $\tau(z_M(y))\geq 1$ and hence $z_M(y)=1$. In particular, for every $0\neq y\in \mathcal P(N)$ and $0\neq x\in \mathcal P(M)$ we have $z_M(x)z_M(y)=z_M(x)\neq 0$. By Comparison Theorem $x$ and $y$ have nonzero subequivalent projections. 

Let $\mathcal F$ be the set of all collections of mutually orthogonal projections $e_i\leq e$ for which there exist mutually orthogonal projections $f_i\in N$ and  $w_i\in M$ satisfying $w_i^*w_i=e_i$ and $w_iw_i^*=f_i$ for all $i$. By above paragraph $\mathcal F\neq \emptyset$ and consider the order $(\mathcal F, \prec )$ given by set inclusion. By Zorn's lemma there is $\{e_i\,|\,i\in I\}$, a maximal collection. Then set $\sum_i e_i=:e'\leq e$, $f:=\sum_if_i\in N$ and note that $e'=w^*w$ and $f=ww^*$, where $w=\sum_i w_i$. To reach our conclusion we only need to argue that $e=e'$. Indeed, if $0\neq e-e'$, by the previous paragraph there exist $0\neq e_0\leq e-e'$, $0\neq f_0\leq 1-f$ and $w_0\in M$ such that $w_0^*w_0=e_0$ and $w_0w_0^*=f_0$. Thus $\{e_i\,|\, i\in I\} \precneqq\{e_i\,|\,i\in I\}\cup \{ e_0\}$ contradicting the maximality of $\{e_i\}$.\end{proof}

With this two facts at hand we are now ready to prove the first main result of the section.

\begin{thrm}\label{virtualprod} Let $\Sigma \leqslant \G$ be countable groups, where $\G$ is icc and $\Sg$ is finite-by-icc. Let $r\in L(\G)$ be a projection and let $P,Q\subset rL(\G)r$ be commuting II$_1$ factors such that $P\vee Q\subseteq rL(\G)r$ has finite index. If  $P\cong^{com}_{L(\G)} L(\Sg)$ 
then there exist a subgroup $\Omega\leqslant C_{\G}(\Sigma)$ such that 
\begin{enumerate}\item \label{a1} $[\G:\Sigma \Omega]<\infty$;
\item \label{a3}$Q\cong^{com}_M L(\Omega)$.
\end{enumerate} 
\end{thrm}
\begin{proof} Since $P\cong^{com}_{L(\G)} L(\Sg)$ there exist projections $p\in P, e\in L(\Sigma)$, a partial isometry $v\in L(\G)$, and  a injective, unital $\ast$-homomorphism $\Phi: pPp\ra eL(\Sigma)e$ so that \begin{enumerate}
\item\label{301} $\Phi(x)v=vx$ for all $x\in pPp$, and 
\item \label{302}$\Phi(pPp)\subseteq eL(\Sigma)e$ is a finite index inclusion of II$_1$ factors.
\end{enumerate}

Denote by $R:=\Phi(pPp)\subseteq eL(\Sigma)e$. Let $T \subseteq R\subseteq eL(\Sigma) e$ be the downward basic construction for inclusion $R\subseteq eL(\Sigma) e$ \cite[Lemma 3.1.8]{Jo81} and let $a\in T'\cap qL(\Sigma) q$ be the Jones' projection satisfying $eL(\Sigma) e=\langle R,a \rangle $ and $a L(\Sigma)a = Ta$. Also note that $[eL(\Sigma)e: R]=[R:T]$. Using equation (\ref{301}) the restriction $\Phi^{-1}: T\ra pPp$ is  an injective $\ast$-homomorphism such that $U:=\Phi^{-1}(T)\subseteq pPp$ is a finite Jones index subfactor and 
\begin{equation}\label{4}
\Phi^{-1} (y)v^*=v^*y\text{, for all }y\in T.
\end{equation}  
Let $\theta':Ta\ra T$ be the $\ast$-isomorphism given by $\theta'(xa)=x$. One can check that $0\neq v^*a$ and let $0\neq w$ be a partial isometry so that $v^*a=w^*_0|v^*a|$. Then $Ta=aL(\Sigma) a$ together with (\ref{4}) shows that $\theta=\Phi^{-1}\circ\theta': aL(\Sigma) a\ra pPp$ is an injective $\ast$-homomorphism satisfying $\theta(aL(\Sigma )a)=U$ and  
\begin{equation}\label{3}
\theta (y)w^*_0=w^*_0y\text{, for all }y\in aL(\Sigma) a.
\end{equation} 

Since $P\vee Q \subseteq rL(\G)r$ is finite index then so is $pPp\vee Q \subseteq pL(\G)p$. Also since $U \subset pPp$ is finite index it follows that $U\vee Q\subseteq pL(\G)p$ is finite index. Since these are factors it follows that $U\vee Q\subseteq pL(\G)p$ admits a finite Pimsner-Popa basis. From construction we have $U\vee Q\subseteq U\vee U'\cap pL(\G)p\subseteq pL(\G)p$ and hence  $U\vee Q\subseteq U\vee U'\cap pL(\G)p$ admits a finite Pimsner-Popa basis. Also since $U\vee Q$ is a factor we have by Proposition \ref{finiteindexbasicprop}(1) that $dim_{\mathbb C}(\mathcal Z(U\vee U'\cap pL(\G)p))<\infty$. As $\mathcal Z(U\vee U'\cap pL(\G)p)=\mathcal Z(U'\cap pL(\G)p)$ we conclude that $dim_{\mathbb C}(\mathcal Z(U'\cap pL(\G)p))<\infty$. Using Proposition \ref{finiteindexbasicprop}(3) for every minimal projection $b\in \mathcal Z(U'\cap pL(\G)p)$ it follows that $(U\vee Q) b\subseteq U\vee (U'\cap pL(\G)p) b$ is a finite index inclusion of II$_1$ factors. Thus there exists $C_b>0$ such that for all $x\in U_+$ and $y\in (U'\cap pL(\G)p)_+$ we have \begin{equation}\label{1201}
\|E_{U\vee Q b}(xyb)\|^2_{2,b}\geq C_b\|xyb\|^2_{2,b}.
\end{equation}

 Since $E_{U\vee Q}(b)=\tau_p(b) p$ we have  $E_{U\vee Q b}(zb)=E_{U\vee Q}(zb) b \tau^{-1}_p(b)$  for all $z\in  U\vee U'\cap pL(\G)p$. Thus for every $x\in U$ and $y\in U'\cap pL(\G)p$ we have $E_{U\vee Q b}(xyb)=E_{U\vee Q}(xyb) b \tau^{-1}_p(b)=xE_{U\vee Q}(yb) b \tau^{-1}_p(b)=xE_{Q}(yb) b \tau^{-1}_p(b)=xE_{Qb}(yb)$. Also since $U$ is a factor one can check that for all $x\in U$ and $y\in U'\cap pL(\G)p$ we have $\|xyb\|^2_2=\|x\|^2_2\|yb\|^2_2$. This further implies that $\|xyb\|^2_{2,b}=\|x\|^2_2\|yb\|^2_{2,b}$. Using these formulas together with (\ref{1201}) we see that $\|x\|^2_2 \| E_{Q b}(yb)\|^2_{2,b}=\|x E_{Q b}(yb)\|^2_{2,b}=\|E_{U\vee Q b}(xyb)\|^2_{2,b}\geq C_b\|xyb\|^2_{2,b}=C_b\|x\|^2_2 \|yb\|^2_{2,b}$ and hence $ \| E_{Q b}(yb)\|^2_{2,b}\geq C_b \|yb\|^2_{2,b}$ for all $y\in U'\cap pL(\G)p$. Hence, $Qb\subseteq (U'\cap pL(\G)p) b$ is a finite index inclusion of II$_1$ factors for every minimal projection $b\in \mathcal Z(U'\cap pL(\G)p)$. 
 
Choose a minimal projection $b\in \mathcal Z(U'\cap pL(\G)p)$ so that $w^*=bw^*_0\neq 0$. Thus (\ref{3}) gives 
 \begin{equation}\label{3'}
\theta (y)w^*=w^*y\text{, for all }y\in aL(\Sigma) a.
\end{equation} 

Notice that $w^*w\in (U'\cap pL(\G)p) b$  and $ww^*\in aL(\Sigma)a'\cap aL(\G)a$. Letting $u\in pL(\G)p$ be a unitary so that $uw^*w=w$, relation (\ref{3'}) entails  \begin{equation}\label{5}u U w^*wu^*=ww^* aL(\Sigma)a.
\end{equation} 

Passing to relative commutants we also have 

\begin{equation}\label{6}
uw^*w (U'\cap pL(\G)p) w^*wu^*=ww^* (aL(\Sigma)a'\cap aL(\G)a) ww^*= ww^* (L(\Sigma)'\cap L(\G)) ww^*
\end{equation}

Altogether, (\ref{5}) and (\ref{6}) imply that 
\begin{equation}\label{7}\begin{split}uw^*w (U\vee (U'\cap pL(\G)p)) w^*wu^*& =ww^* (a L(\Sigma)a \vee  (aL(\Sigma)a'\cap aL(\G)a)) ww^*\\& = ww^* ( L(\Sigma) \vee  (L(\Sigma)'\cap L(\G))) ww^*.\end{split}
\end{equation} 

Since from assumptions $pPp \vee Qp =p(P\vee Q)p\subseteq pL(\G)p$ is a finite index and $U\subseteq pPp$ is finite index it follows that $U \vee Qp \subseteq pL(\G)p$ is finite index as well. Also notice that
$U\vee Qp \subseteq U\vee (P'\cap rL(\G)r)p = U\vee (pPp'\cap pL(\G)p)\subseteq U\vee (U'\cap pL(\G)p)\subseteq pL(\G)p$. Thus $U\vee (U'\cap pL(\G)p)\subseteq pL(\G)p$ is finite index. Combining with (\ref{7}) we obtain that $ww^*( L(\Sigma) \vee  (L(\Sigma)'\cap L(\G))) ww^*\subseteq ww^* L(\G)ww^*$ is a finite index inclusion of II$_1$ factors. Using Theorem \ref{fromrelcomtocomgroups} there exists a subgroup $\Omega<\Lambda$ such that $[\Sigma,\Omega]=1$ and $[\G:\Sigma\Omega]<\infty$.  Since $\G$ is icc if follows that $\Sigma,\Omega$ are icc as well; in particular, $L(\Sigma)$, $L(\Omega)$ are II$_1$ factors.
Using Lemma \ref{intinsubfactor} there exist unitaries $v_1\in U'\cap pL(\G)p$, $u_2\in L(\Sigma)'\cap L(\G)$ such that $v_1 w^*w v_1^*=q_1\in Qb$ and $u^*_2 ww^* u_2=q_2\in L(\Omega)$. Denoting by $t:=u_2^*uv^*_1$ then relation (\ref{6}) can be rewritten as
\begin{equation}\label{8}
t q_1 (U'\cap pL(\G)p) q_1 t^*=q_2 (L(\Sigma)'\cap L(\G))q_2. 
\end{equation}  
Since $[\G:\Sigma\Omega]<\infty$ then $q_2 L(\Sigma \Omega) q_2 \subseteq q_2L(\G)q_2$ has finite index. Since $L(\Omega)\subseteq L(\Sigma)'\cap L(\G)$ it follows that $q_2 L(\Sigma \Omega)q_2\subseteq q_2L(\Sigma )\vee L(\Sigma)'\cap L(\G)q_2$ is finite index as well. Therefore $q_2L(\Omega)q_2\subseteq q_2L(\Sigma)'\cap L(\G)q_2$ is finite index inclusion of II$_1$ factors. By Lemma \ref{finiteindeximage} there exist $r_i\leq q_2$, $w_1\in q_2L(\Sigma)'\cap L(\G)q_2$,  and a $*$-isomorphism  $\phi': r_1L(\Sigma)'\cap L(\G)r_1 \ra B\subseteq r_2L(\Omega)r_2$ such that 
\begin{enumerate}
\item \label{10'}$\phi'(x)w_1=w_1 x$ for all $x\in r_1L(\Sigma)'\cap L(\G)r_1$
\item\label{10} $[r_2L(\Omega)r_2:B]<\infty$
\end{enumerate}
Using Lemma \ref{intinsubfactor}, relation (\ref{8}), and perturbing more the unitary $t$ we can assume there exists a projection  $q_3\in Q$ such that  $q_3 b\leq q_1$ and 
\begin{equation}\label{8'}
t q_3 (U'\cap pL(\G)p) q_3 b t^*=r_1 (L(\Sigma)'\cap L(\G))r_1. 
\end{equation} 
 
Consider the $\ast$-isomorphism $\Psi':q_3Qq_3\ra tq_3 Qq_3b t^*$ given by
$\Psi'(x)=txb t^*$ and let $\Psi=\phi'\circ\Psi': q_3Qq_3 \ra r_2L(\Omega)r_2$. Using (\ref{10'}) above for every $x\in q_3Qq_3$ we have $\Psi(x) w_1t=\phi'(\Psi'(x))w_1t = w_1 \Psi'(x) t= w_1 txb  = w_1tb x$. Next we argue that $w_1tb\neq 0$. If $w_1tb= 0$ then  $w_1tbq_1 t^* = 0$ and hence $w_1q_2 = 0$. Thus $w_1=w_1 r_1= w_1r_1q_2=0$, a contradiction.  So letting $w$ to be the isometry in the polar decomposition of $w_1tb$, $q:=q_3$ and $f:=r_2$ we get that $\Psi: qQq\ra fL(\Omega)f$ a injective, unital $\ast$-homomorphism so that    $\Psi(x)w=wx$ for all $x\in qQq$. Moreover since $Qb\subseteq q_1(U'\cap pL(\G)p) q_1$ is finite index then using (\ref{302}) above and (\ref{8'}) one gets that $\Psi(qQq)\subseteq r_2 L(\Omega)r_2$ has finite index. Altogether these show that $Q\cong^{com}_{L(\G)} L(\Omega)$.  \end{proof}

We end this section presenting the second main result. This roughly asserts that tensor product decompositions of group von Neumann algebras whose factors are commensurable with subalgebras arising commuting subgroups can be ``slightly perturbed'' to tensor product decompositions arising from the actual direct product decompositions of the underlying group. The proof uses the factor framework in an essential way and it is based on arguments from \cite[Proposition 12]{OP03} and \cite[Theorem 4.14]{CdSS15} (see also \cite[Theorem 6.1]{DHI16}).  

\begin{thrm}\label{fromvirtualtoproduct}  Let $\G$ be an icc group and assume that  $M=L(\G)=M_1\bar\otimes M_2$, where $M_i$ are diffuse factors. Also assume there exist commuting, non-amenable, icc subgroups $\Sigma_1,\Sigma_2<\G$ such that $[\G:\Sigma_1 \Sigma_2]<\infty$, $M_1\cong^{com}_M L(\Sg_1)$, and $M_2\cong^{com}_M L(\Sg_2)$.  Then there exist a group decomposition $\G=\G_1\times \G_2$, a unitary $u\in M$ and $t>0$ such that $M_1=uL(\G_1)^tu^*$ and $M_2=uL(\G_2)^{1/t}u^*$.  
 \end{thrm}

\begin{proof}
In particular, we have $L(\Sigma_1)\prec_M M_1 $. Since $M=M_1\bar\otimes M_2$ then proceeding as in the proof of  \cite[Proposition 12]{OP03}  there exist a scalar $\mu>0$ and a partial isometry $v\in M$ satisfying $p=vv^*\in M^{1/\mu}_2$, $q=v^*v\in L(\Sigma_1)'\cap M$ and
\begin{equation}\label{4.1'}vL(\Sigma_1) v^*\subseteq M^\mu_1p.\end{equation} 

Let $\Omega_2=\{ \g \in\G \,:\, |\mathcal O_{\Sigma_1}(\g) |<\infty \}$ and notice that $ \Sigma_2 \leqslant \Omega_2$. As $[\G:\Sg_1\Sg_2]<\infty$  then $[\G: \Omega_2\Sg_1]<\infty$. Letting $\Omega_1=C_{\Sigma_1}(\Omega_2)$, one can check that $\Omega_1,\Omega_2<\La$ are commuting, non-amenable, icc subgroups  satisfying $[\G:\Omega_1 \Omega_2]<\infty$ and $[\Sg_1:\Omega_1]<\infty$. Also notice that $L(\Sigma_1) '\cap M \subseteq L(\Omega_2)$ and  by relation (\ref{4.1'}) we have $uL(\Omega_1) u^*\subseteq M^\mu_1p$. Since $L(\Omega_2)$ and $M^{1/\mu}_2$ are factors then as in the proof of \cite[Proposition 12]{OP03}, one can find partial isomoetries $w_1, \ldots,w_m\in L(\Omega_2)$ and  $u_1,\ldots, u_m \in M^{1/\mu}_2$ satisfying $w_i{w_i}^*=q'\leq q$, ${u_i}^*u_i=p'=uq'u^*\leq p$ for all $i$ and $\sum_j {w_j}^*w_j=1_{L(\Omega_2)}$, $\sum_j u_j{u_j}^*=1_{M^{1/\mu}_2}$. Combining with the above, one can check $u=\sum_j u_j vw_j\in M$ is a unitary satisfying $u L(\Omega_1)u^*\subseteq M^\mu_1$. Since $M=M_1^{\mu}\bar\otimes M_2^{1/\mu}$ then  \begin{equation}\label{4.2'} u(L(\Omega_1)'\cap M)u^*\supseteq M^{1/\mu}_2.\end{equation}

Let $\Theta_2=\{ \la \in\La \,|\, |\mathcal O_{\Omega_1}(\la) |<\infty \}$ and $\Theta_1=C_{\Omega_1}(\Theta_2)$ and as before it follows that $\Theta_1,\Theta_2 <\La$ are commuting, non-amenable, icc subgroups  such that $[\G:\Theta_1 \Theta_2]<\infty$ and $[\Sg_1:\Theta_1]<\infty$.  Moreover, by (\ref{4.2'}) we have $uL(\Theta_2)u^*\supseteq M^{1/\mu}_2$. Since $M=M^\mu_1\bar \otimes M^{1/\mu}_2$ by \cite[Theorem A]{Ge96} we have $uL(\Theta_2)u^*=B\bar \otimes M^{1/\mu}_2$, where $B\subseteq M^\mu_1$ is a factor. But relation implies that $uL\Sigma_2u^* \prec M^{1/\mu}$ and since $[\Omega_2:\Sg_2]<\infty$ it follows that $uL(\Omega_2)u^*\prec M_2^{1/\mu}$. Since $B\subseteq uL(\Omega_2)u^*$ it follows that $B\prec M_2^{1/\mu}$. However since $B\subseteq M_1^{\mu}$ and $M=M_1^{\mu}\bar\otimes M_2^{1/\mu}$ this forces that $B$ has an atomic corner. As $B$ is a factor we get $B=M_k(\mathbb C)$, for some $k\in \mathbb N$. Altogether, we have $uL(\Theta_2)u^*= M^{t}_2$ where $t=k/\mu$. Since $M=M^{1/t}_1\bar\otimes M^t_2$ we also get $u(L(\Theta_2)'\cap M)u^*= M^{1/t}_1$. Let $\G_1=\{ \la \in\La \,|\, |\mathcal O_{\Theta_2}(\la) |<\infty \}$ and since $\Theta_2$ is icc it follows that $ \G_1\cap\Theta_2=1$.  By construction, $uL(\G_1)u^*\supseteq u(L(\Theta_2)'\cap M)u^*= M^{1/t}_1$; hence, \cite[Theorem A]{Ge96} gives $uL(\G_1)u^*=A\bar \otimes u(L(\Theta_2)'\cap M)u^*$ for some $A \subseteq uL(\Theta_2)u^*$. In particular, we have $A=u L(\G_1)u^*\cap u L(\Theta_2)u^* = \mathbb C 1$ and hence $uL(\G_1)u^*=u(L(\Theta_2)'\cap M) u^*$. Letting $\G_2=\Theta_2$ it follows that $\G_1$ and $\G_2$ are commuting, non-amenable subgroups of $\G$ such that $\G_1\cap\G_2=1$, $\G_1\G_2=\G$, $uL(\G_1)u^*=  M^{1/t}_1$, and $uL(\G_2)u^*= M^{t}_2$.\end{proof}


\section{Finite-step Extensions of Amalgamated Free Product (or Poly-amalgam) Groups}\label{sec:FiniteStepExt} Fixing $\Cal C$ a class of groups we recall next the finite-step extensions groups over $\Cal C$. First let $Quot_{1}(\mathcal C)=\mathcal C$. For any integer $n\geq 2$, we say that a group $\G\in Quot_n(\Cal C)$ if there exist: 
\begin{enumerate}
\item groups  $\Gamma_k$, for all $k\in \overline{0,n}$, such that $\G_0=1$ and $\G=\G_n$, and 
\item epimorphisms $ \pi_k:\G_{k}\rightarrow \G_{k-1}$ such that $ker(\pi_k)\in \Cal C$, for all $k\in \overline{1,n}$.
\end{enumerate}

If $n$ is the minimal integer satisfying conditions (1)-(2) above then $\G$ is called a $n$-step extension of groups in $\mathcal C$. We denote by $Quot(\mathcal C):=\cup_{n\in\mathbb N} Quot_n(\mathcal C)$, the class of all \emph{finite-step extension of groups in $\mathcal C$}. We also denote by $Quot_n(\mathcal C)^c$ (resp $Quot(\mathcal C)^c$) the class of all groups commensurable with groups in $Quot_n(\mathcal C)$ (resp $Quot(\mathcal C)$). 

Below we recall some useful elementary algebraic properties of group extensions. The proofs are omitted as they are either straightforward or already contained in \cite[Lemmas 2.9-2.10]{CIK13} and \cite[Proposition 3.3]{CKP14}.

\begin{prop}\label{quot} The following properties hold: 
\begin{enumerate}
\item If $\G\in Quot_n(\Cal C)$ and $p: \La \ra \G$ is a epimorphism such that $ker(p)\in\Cal C$ then $\La \in Quot_{n+1}(\Cal C)$; 
\item If $\G_i \in Quot_{n_i}(\mathcal C)$ for all $i\in \overline{1, k}$ then $\G_1 \times \cdots \times \G_k \in Quot_{n_1+\cdots +n_k}(\Cal C)$;
\item If $\Cal C=\Cal C^c$ then $Quot_n(\Cal C)^c= Quot_n(\Cal C)$ and $Quot(\Cal C)^c= Quot(\Cal C)$;
\item Let $\G\in Quot_n(\Cal C)$, let $ \pi_k:\G_{k}\rightarrow \G_{k-1}$ be the epimorphisms satisfying the previous definition, and let $p_n=\pi_2\circ\cdots \circ\pi_n:\Gamma_n\rightarrow\Gamma_1$. Then the following hold:
\begin{enumerate} 
\item If $\La<\G_1$ is a subgroup such that $\La\in \Cal C$ then $p_n^{-1}(\La)\in Quot_n(\Cal C)$;
\item $\ker(p_n)\in Quot_{n-1}(\Cal C)$; moreover, if $\Cal C$ is closed under commensurability and $\La<\G$ is a subgroup such that $p_n(\La)$ is finite  then $p_n^{-1}(p_n(\La))\in Quot_{n-1}(\Cal C)$. 
\end{enumerate}

\item  $\G\in Quot_n(\mathcal C)$ if and only if there exists a composition series $1= \La_0\lhd\La_1\lhd \cdots \lhd \La_n$ such that $\G=\La_n$ and the factors $\La_i/\La_{i-1}\in \mathcal C$ for all $i\in \overline{0,n-1}$; in other words, $\G$ is a poly-$\mathcal C$ group.      
\end{enumerate}
\end{prop}

\begin{notation} \label{notation:TreeGroups}We denote by $\mathcal{T}$ all groups $\G$ belonging to one of the following families:\begin{enumerate} 

\item amalgamated free products $\G =\G_1\ast_\Sigma\G_2$ satisfying $[\G_1:\Sigma]\geq 2$, $[\G_2:\Sigma]\geq 3$, and $\Sigma\in \mathcal C_{\rm rss}$ (see  \cite[Definition 2.7]{CIK13} for $\mathcal C_{\rm rss}$); note this includes all finite $\Sg$;

\item HNN-extensions $\G ={\rm HNN}(\La,\Sg,\theta)$  satisfying $\Sg\neq \La\neq \theta(\Sg)$ and $\Sigma\in \mathcal C_{\rm rss}$.
\end{enumerate}\end{notation}

Next we will describe all possible nontrivial direct product decompositions of finite step extensions of (tree) groups $\mathcal{T}$. These results will be used later on in the proofs of the main results of Section \ref{sec:classtensorgroupfactor}. First we treat the amalgamated free product and HNN-extension case. 
\begin{lem}\label{amalgamprodecomp} The following hold:
\begin{enumerate} 
\item Let $\G=\G_1\ast_\Sigma \G_2$. If $\G=\La_1\times \La_2$ then one can find $\sigma\in \mathfrak S_2$  satisfying $\Sg=\La_{\sigma(1)} \times \Sg_0$, $\G_1=\La_{\sigma(1)} \times \G^0_1$, $\G_2= \La_{\sigma(1)}\times \G^0_2$, and $\La_{\sigma(2)}=\G^0_1\ast_{\Sg_0} \G^0_2$.
\item Let $\G={\rm HNN}(\Psi,\Sigma,\theta) $ be non-degenerate. If $\G=\La_1\times \La_2$ then one can find $\sigma \in \mathfrak S_2$  such that $\Sg=\La_{\sigma(1)} \times \Sg_0$, $\Psi=\La_{\sigma(1)} \times \Psi_0$. Moreover $\theta_{\La_{\sg(1)}} =ad(t_1)$ for some $t_1\in \La_{\sg(1)}$  and $\La_{\sigma(2)}={\rm HNN}(\Psi_0,\Sg_0, \theta_{| \Sg_0} )$. 
\end{enumerate}
\end{lem}
\begin{proof} (1) Since $\G=\La_1\times \La_2$ then using Theorem \ref{intertwiningincore1} and  \cite[Lemma 2.2]{CI17} there  is $h\in \G$ so that $[\La_{\sigma(1)}: h\Sigma h^{-1}\cap\La_{\sigma(1)}]<\infty$. So conjugating by $h$ we can assume that $[\La_{\sigma(1)}: \Sigma \cap \La_{\sigma(1)}]<\infty$. Also passing to a finite index subgroup we can assume  $\Sigma \cap \La_{\sigma(1)}$ is normal in $\G$. Therefore we have $\G/(\Sigma\cap\La_{\sigma(1)})= (\G_1/(\Sigma\cap \La_{\sigma(1)})) \ast_{\Sigma /(\Sigma\cap \La_{\sigma(1)})} (\G_2 /(\Sigma\cap \La_{\sigma(1)})) = \La_{\sigma(1)}/(\Sigma \cap \La_{\sigma(1)}) \times \La_{\sigma(2)}$. Since $\La_{\sigma(1)}/(\Sigma \cap\La_{\sigma(1)})$ is finite then \cite[Theorem 10]{KS70}  implies that $\La_{\sigma(1)}/(\Sigma \cap \La_{\sigma(1)})\leqslant\Sigma /(\Sigma\cap\La_{\sigma(1)})$  and thus $\La_{\sigma(1)}\leqslant \Sigma $. Hence $\Sg=\La_{\sigma(1)}\times \Sg_0$ and $\G_i= \La_{\sigma(1)} \times \G^0_i$ for all $i\in\overline{1,2}$, and  $\La_{\sigma(2)}=\G^0_1\ast_{\Sg_0} \G^0_2$. 

(2) Since $\G=\La_1\times \La_2$ then using Theorem \ref{intertwiningincore2} and proceeding as in the previous case  we can assume that  $\Omega:=\Sigma \cap \La_{\sigma(1)}$ is normal in $\G$ and has finite index in $\La_{\sigma(1)}$. Let $t\in \G$ be so that $\theta(x)=txt^{-1}$ for all $x\in \Sg$. By normality we have $\theta(\Omega)=\Omega$ and moreover one can check $\G/\Omega={\rm HNN}(\Psi/\Omega, \Sigma/\Omega, \tilde \theta)$, where $\tilde\theta(x\Omega)= \theta(x)\Omega
$.  Also we have $\G/\Omega = \La_{\sg(1)}/\Omega \times \La_{\sg(2)}$ and since $ \La_{\sg(1)}/\Omega$ is finite it follows from  \cite[Proposition 3]{dC09} $ \La_{\sg(1)}/\Omega\leqslant \Sg/\Omega$; in particular, $\Omega =\La_{\sg(1)}\leqslant \Sigma$. Hence $\Sg=\La_{\sigma(1)} \times \Sg_0$, $\Psi=\La_{\sigma(1)} \times \Psi_0$. Writing $t=t_1t_2$ with $t_i\in \La_{\sg(i)}$ one can check that $\theta(x)=t_2xt_2^{-1}$ for all $x\in \Sg_0$ and also $\theta(\Sigma_0)<\Psi_0$. In addition one can check that 
$\theta_{|\La_{\sg(1)}} =ad(t_1)$ and also $\La_{\sg(2)}= {\rm HNN}(\Psi_0,\Sg_0,\theta_{| \Sg_0})$. \end{proof}

\begin{thrm}\label{prodstructureextensionamalagam} Denote by $\mathcal D=\mathcal C_{\rm rss} \cup (\text{finite-by-}\mathcal{T})^c$. Let $\G \in Quot^c_n(\mathcal{T})$ such that $\G=\La_1\times \La_2$. Then for every $i\in\overline{1,2}$ there exist there exists $k_i\in \overline{1,n}$ such that $\La_i\in Quot^c_{k_i}(\mathcal D)$. \end{thrm}

\begin{proof} We proceed by induction on $n$. As the case $n=1$ follows immediately from Lemma \ref{amalgamprodecomp} we only need to show the induction step. Since $\G \in Quot^c_n(\mathcal{T})$ there exist: 
\begin{enumerate}
\item groups  $\Gamma_k$, for all $k\in \overline{0,n}$, such that $\G_0=1$ and $\G$ is commensurable to  $\G_n$, and 
\item epimorphisms $ \pi_k:\G_{k}\rightarrow \G_{k-1}$ such that $\ker(\pi_k)\in \Cal {T}$, for all $k\in \overline{1,n}$.
\end{enumerate}
Consider the epimorphism $\rho= \pi_2\circ \cdots \circ \pi_n:\G_n \ra \G_1$. Since  $\G=\Lambda_1\times \Lambda_2$, then $\Gamma_1=\rho(\Lambda_1)\times\rho(\Lambda_2) $.  Since $\Gamma_1 \in\mathcal{T}$, either $\Gamma_1=\Theta_1 \ast_\Delta \Theta_2 $ or $\Gamma_1={\rm HNN}(\Theta, \Delta,\theta) $.

We first  treat the case where $\Gamma_1\in \mathcal{T} $ is an  amalgamated free product $\Theta_1 \ast_\Delta \Theta_2 $ as described in Notation \ref{notation:TreeGroups}.   By Lemma \ref{amalgamprodecomp} we can assume that
$\Delta=\Omega \times \Delta_0$,  $\Theta_1= \Omega \times \Theta^0_1$, $\Theta_2= \Omega \times \Theta^0_2$, $\rho(\La_1)=\Omega$, and $\rho(\La_2)=\Theta^0_1\ast_{\Delta_0} \Theta^0_2$.  Thus $\rho^{-1} (\Omega)=\La_1\times \La^0_2$ and $\rho^{-1} (\Theta^0_1\ast_{\Delta_0} \Theta^0_2)=\La^0_1\times \La_2$, for some subgroups $\La^0_i\leqslant \La_i$. Observe that $\La^0_1 \times \La^0_2 =\ker (\rho)$ and hence $\La^0_i\lhd \La_i$ are normal. Combining with the above we have \begin{equation}\label{quot1}
\La_1/\La^0_1 \cong \Omega,\quad   \La_2/\La^0_2 \cong \Theta^0_1\ast_{\Delta_0} \Theta^0_2.\end{equation} 
Since  $\ker(\rho)\in Quot_{n-1}(\mathcal{T})$ and $\La_i$ are icc it follows by the induction hypothesis that either $\La^0_i\in Quot^c_{k_i}(\mathcal D)$  for some $k_i\in \overline{1,n-1}$ or $\La^0_i=1$.   Also, as $\Delta \in \mathcal C_{\rm rss}$ we have either 
\begin{enumerate}\item [a)] $\Omega$ is finite and $\Delta_0 \in \mathcal C_{\rm rss}$, or 
\item [b)] $\Delta_0$ is finite and $\Omega \in \mathcal C_{\rm rss}$.
\end{enumerate}

If a) holds then by (\ref{quot1}) $\La_1$ is commensurable with $\La_1^0$, and hence $\La_1\in Quot^c_{k_1}(\mathcal D)$ as desired. When $\La^0_2=1$ (\ref{quot1}) already gives $\La_2\in Quot_1(\mathcal{T})$, so assume $\La^0_2$ is infinite. Since $\La_2^0\lhd \La_2$ is normal then using the inductive hypothesis there is a finite index subgroup $A\leqslant \La_2^0$ which is normal in $\La_2$ and satisfies $A\in Quot_{k_2}(\mathcal D)$. By Proposition \ref{quot} (5), $A$ is a poly-$\mathcal D$ group of length $k_2$. Using the isomorphism theorem we have $\La_2/\La^0_2= (\La_2/A)/(\La_2^0/A)$ and from (\ref{quot1}) it follows $\La_2/A$ is finite-by-$\mathcal C_{rss}$ and hence in $\mathcal D$. Thus $\La_2$ is poly-$\mathcal D$ of length $k_2+1$ and by Proposition \ref{quot}(5) $\La_2\in Quot_{k_2+1}(\mathcal D)$. 

Assume b) holds. If $\La^0_1=1$ then from (\ref{quot1}) $\La_1\in \mathcal C_{\rm rss}\subset Quot_1(\mathcal D)$.  Also since $\Delta_0$ is finite it follows form  (\ref{quot1}) that $\La_2/\La_2^0 = \Theta^0_1\ast_{\Delta_0} \Theta^0_2\in \mathcal C_{\rm rss}$. Using the inductive hypothesis there exists a finite index subgroup $A\leqslant \La^0_2$ that is normal in $\La_2$ and satisfies $A\in Quot_{k_2}(\mathcal D)$. Proceeding as in the previous case we have $\La_2 \in Quot_{k_2+1}(\mathcal D)$, as desired. When both $\La^0_i$ are infinite one can show in a similar manner $\La^0_i\in Quot_{k_i+1}(\mathcal D)$, for all $i\in\overline{1,2}$. When $\La^0_2=1$ one can show that $\La^0_1\in Quot_{k_1+1}(\mathcal D)$ and $\La_2= \Theta^0_1\ast_{\Delta_0} \Theta^0_2\in \mathcal C_{\rm rss}\subset Quot_1(\mathcal D)$. 

Now if $\Gamma={\rm HNN}(\Theta, \Delta, \theta) $, Lemma \ref{amalgamprodecomp}  once again allows us to assume that $\Delta=\Omega\times \Delta $, $\Theta=\Omega\times \Theta_0 $, $\rho(\Lambda_1)=\Omega $, and $\rho(\Lambda_2)={\rm HNN}(\Theta_0,\Delta_0, \theta_{| \Delta_0}) $.  The result now follows by proceeding as in the case where  $\Gamma_1$ was assumed to be an amalgamated free product. 
\end{proof}

\begin{prop}\label{productstructureiterationamalgam} Let $\G= (\G_1\ast_{\Sigma_1}\G_2)\ast_{\Sigma_2}\G_3)\ast_{\Sigma_3}\G_4)\cdots )\ast_{\Sigma_{n-1}} \G_{n}$ be an iterated amalgamated free product and assume that for each $i=\overline{1,n}$ the group $\Sg_i$ is direct product indecomposable. Assume that $\G=\La_1\times\La_2$ for $\La_i$ infinite groups. Then there exist direct product decompositions $\G_i =\Omega_i\times \Sigma_i$  and a subset $F\subseteq \overline{3,n}$ satisfying the following properties:  
\begin{enumerate}
\item Denoting by $F^c= \overline{1,n}\setminus F$ we have $$\La_1 = \mathcal B(F^c)=\Sigma_1 \ast (\ast_{i\in F^c}  \Omega_i),\qquad   \La_2=\mathcal A(F)= \ast_{i\in F}  \Omega_i;$$

\item For all $i\in\overline{1,n}$ we have $$\Sigma_{i}=\begin{cases} \mathcal A(F\setminus\overline{n,i+1}),\text{ if } i+1 \in F^c \setminus\overline{n,i+2} \\ \mathcal B(F^c\setminus \overline{i+2}), \text{ if } i+1 \in F \setminus\overline{n,i+2}\end{cases}.$$ 
\end{enumerate}
 \end{prop}

\begin{proof} Follows by applying successively Lemma \ref{amalgamprodecomp}. \end{proof}


\section{Classification of Tensor Product Decompositions of Group II$_1$ Factors}\label{sec:classtensorgroupfactor}

In \cite{DHI16} it was discovered a new classification result in the study of tensor product decompositions of II$_1$ factors. Precisely, whenever $\G$ is an icc group that is measure equivalent to a direct product of non-elementary hyperbolic groups then all possible tensor product decompositions of the II$_1$  factor $L(\G)$ arise from the canonical direct product decompositions of the underlying group $\G$. The same result holds when $\G$ is poly-hyperbolic with non-amenable factors in the composition series \cite{dSP17}.    

In this section we provide yet several new examples of groups $\G$ for which this classification of tensor product decompositions of $L(\G)$ holds. These include many natural examples $\G$ of amalgamated free product groups (Theorem \ref{tensordecompamalgam}), poly-amalgam groups (Theorem \ref{tensordecomppolyamalgams}), or graph product groups (Theorem \ref{tensordecompgraphprod}). These results  are obtained by combining  the classification of normalizers of subalgebras in AFP von Neumann algebras from \cite{Io12,Va13} or group-extensions von Neumann algebras \cite{CIK13} with the techniques on commensurable algebras from Section \ref{sec:commensurablealgebras}.  

\subsection{Amalgamated free products}

\begin{thrm}\label{tensordecompamalgam} Let $\G=\G_1\ast_{\Sigma}\G_2$ be an icc group with $[\G_1:\Sigma]\geq 2$ and $[\G_2:\Sigma]\geq 3$ and assume that any corner of $L(\Sg)$ is virtually prime. Assume that $M_1\bar\otimes M_2 =L(\G)$ for diffuse $M_i$'s. Then there exist direct product decompositions $\Sg=\Omega \times \Sg_0$,  $\G_1= \Omega \times \G^0_1$, $\G_2= \Omega \times \G^0_2$ with  $\Sg_0$ finite. Moreover there exist a unitary $u\in L(\G)$, $t>0$ and a permutation $\sigma$ of $\{1,2\}$ such that  
\begin{equation}
M_{\sigma(1)}= u L(\Omega)^t u^*\quad \text{ and }\quad  M_{\sigma(2)} = u L(\G_1^0\ast_{\Sg_0} \G_2^0)^{1/t} u^*.
\end{equation}
\end{thrm}

\begin{proof}Since $M_1\bar\otimes M_2=L(\G)$, by Corollary \ref{intertwiningcoregroupsafphnn} we can assume $M_{\sigma(1)} \prec L(\Sigma)$. Since any corner  of $L(\Sigma)$ is virtually prime then by Lemma \ref{intertwiningdichotomy} we must have  $M_{\sigma(1)}\cong_M^{com} L(\Sigma)$, and further applying Theorems \ref{virtualprod} and \ref{fromvirtualtoproduct} there exist infinite groups $\La_i$ so that $\G=\La_1\times \La_2$. Thus the desired conclusion follows by using Lemma \ref{amalgamprodecomp}.  \end{proof}

\begin{cor}\label{tensordecompamalgamiterated} Let $\G= (\G_1\ast_{\Sigma_1}\G_2)\ast_{\Sigma_2}\G_3)\cdots )\ast_{\Sigma_{n-1}} \G_{n}$ be an iterated amalgamated free product satisfying $[(\G_1\ast_{\Sigma_1}\G_2)\ast_{\Sigma_2}\G_3)\cdots \ast_{\Sg_{n-2}}\G_{n-1}:\Sigma_{n-1}]\geq 2$ and $[\G_n:\Sigma_{n-1}]\geq 3$. Also for each $i\in\overline{1,n}$ assume $\Sg_i$ is direct-product indecomposable and $L(\Sg_i)$ virtually prime.  Assume that $M_1\bar\otimes M_2 =L(\G)$ for diffuse $M_i$'s. Then there exist direct-product decompositions $\G_i =\Omega_i\times \Sigma_i$  and a subset $F\subseteq\overline{3,n} $ satisfying the following properties:  
\begin{enumerate}
\item  $\G = \mathcal A(F) \times  \mathcal B(F^c)$, where $F^c= \overline{1,n}\setminus F$ and  $$ \mathcal A(F)= \ast_{i\in F}  \Omega_i \qquad \mathcal B(F^c)=\Sigma_1 \ast (\ast_{i\in F^c}  \Omega_i)  ;$$

\item For all $i\in \overline{1,n}$ we have $$\Sigma_{i}=\begin{cases} \mathcal A(F\setminus\overline{n,i+1}),\text{ if } i+1 \in F^c \setminus \overline{n,i+2} \\ \mathcal B(F^c\setminus \overline{n,i+1}), \text{ if } i+1 \in F \setminus \overline{n,i+2}\end{cases}.$$ 
\end{enumerate}
Moreover, there is a unitary $u\in L(\G)$, $t>0$, and a permutation $\sigma$ of $\{1,2\}$ such that  
\begin{equation}
M_{\sigma(1)}= u L(\mathcal A(F))^t u^*\quad \text{ and }\quad  M_{\sigma(2)} = u L(\mathcal B(F^c))^{1/t} u^*.
\end{equation}
\end{cor}
\begin{proof} The proof follows from Theorem \ref{tensordecompamalgam} and Proposition \ref{productstructureiterationamalgam}.\end{proof}

\begin{Remarks}\label{rk1}  Theorem \ref{tensordecompamalgam} presents a situation when a faithful von Neumann algebraic counterpart of Lemma \ref{amalgamprodecomp} could be successfully obtained. However, if one drops the primeness assumption on $L(\Sigma)$, the conclusion of the theorem is no longer true. Precisely, there exists icc amalgam $\G=\G_1\ast_\Sigma \G_2$ whose group factors $L(\G)$ admit nontrivial tensor product decompositions while $\G$ is direct-product indecomposable. For example, take groups $\Sigma <\Omega$ satisfying the following conditions: 
\begin{enumerate} \item [i)] for each finite $E\subset \Omega$ there are $\g,\lambda \in \Sigma$ so that $[\g,E]=[\la,E]=1$ and $\g\la\neq \la\g$ \cite{Jo98}; 
\item [ii)] for each $\g\in \Sg$ there is $\la\in \Omega$ so that $\g\la\neq \la\g$.
\end{enumerate}
Concrete such examples are $\Sigma = \oplus_{\mathfrak S_{\infty}} H < \Omega= \cup_{n\in \mathbb N}(H\wr \mathfrak S_n)$, where $H$ is any icc group and $\mathfrak S_{\infty}$ is the group of finite permutations of $\mathbb N$.

Then the inclusion  $\Sigma <\G=\Omega\ast _\Sg \Omega$ still satisfies i) and by \cite[Proposition 2.4]{Jo98} $L(\G)$ is McDuff so $L(\G)=L(\G)\bar\otimes \Cal R$, where $\Cal R$ is the hyperfinite factor. On the other hand, combining Lemma \ref{amalgamprodecomp} with ii) one can see $\G$ is indecomposable as a direct product. 

Thus it remains open to investigate other natural conditions on $L(\Sg)$ that will insure a statement similar to the conclusion of Theorem \ref{tensordecompamalgam}. For instance recycling the same arguments from the proof of Theorem \ref{tensordecompamalgam} one can easy show this is still the case if one requires $\Sigma\in Quot(\mathcal{T} \cup \mathcal C_{rss})^c$. However little is known beyond those examples and a future investigation in this direction as potential of revealing new interesting technology.
\end{Remarks}
To this end we notice the following counterpart of Theorem \ref{tensordecompamalgam} for HNN-extensions. 
\begin{thrm}\label{tensordecomphnn} Let $\G={\rm HNN}(\La,\Sigma,{\theta})$ be an icc group and assume that any corner of $L(\Sg)$ is virtually prime. Assume that $M_1\bar\otimes M_2 =L(\G)=M$, for diffuse $M_i$'s. Then there exist decompositions $\Sg=\Omega \times \Sg_0$ with $\Sg_0$ finite, $\La= \Omega \times \La_0$. Also, there is $\om\in \Om$ such that  $\theta= ad(\om)_{|\Om} \times \theta_{|\Sg_0} :\Om \times \Sg_0 \ra \Om \times \La_0$ and also $\G = \Omega \times {\rm HNN}(\La_0,\Sg_0,\theta_{|\Sg_0})$. In addition, one can find a unitary $u\in L(\G)$ and $t>0$ such that  
\begin{equation}
M_1= u L(\Omega)^t u^*\quad \text{ and }\quad  M_2 = u L({\rm HNN}(\La_0,\Sg_0, \theta_{| \Sg_0} ))^{1/t} u^*.
\end{equation}
\end{thrm}

\begin{proof} Follows using Corollary \ref{intertwiningcoregroupsafphnn}, Lemma \ref{amalgamprodecomp} (2), and recycling the same arguments from the proof of Theorem \ref{tensordecompamalgam}.\end{proof}

\subsection{Extensions of amalgamated free products} 

\begin{notation}\label{1} Assume that $\La,\G$ are groups and let $\pi:\La\ra \G$ be an epimorphism. Then following \cite{CIK13} we denote by $\Delta^\pi: L(\La)\ra L(\La)\bar\otimes L(\G)$ the $\star$-homomorphism induced by $\Delta(u_\la)= u_\la\otimes v_{\pi(\la)}$ where $\{u_\la | \la\in \La\}$  and $\{v_\g |\g\in \G\}$ are the canonical unitaries of  $L(\La)$ and $L(\G)$, respectively.       
\end{notation}

\begin{thrm}\label{intertquotient} Assume Notation \ref{1} above and suppose that  $\G=\G_1\ast_\Sigma\G_2$ or  $\G={\rm HNN} (\Lambda,\Sigma,\theta)$. Let $A_1,A_2\subset L(\La)=M$ be diffuse subalgebras such that $A_1\vee A_2\subset M$ has finite index. Then there exists $i=1,2$ such that $A_i\prec L(\pi^{-1}(\Sigma))$. 
\end{thrm}

\begin{proof} We treat only the case of amalgamated free product as the other one follows similarly. Let $\pi:\Lambda\to \Gamma $ be an epimorphism and consider the embeddings $\Delta^\pi(A_1), \Delta^\pi(A_2)\subset \Delta^\pi(M)\subset L(\Lambda)\bar\otimes L(\Gamma) =\widetilde{M}$.  Now we may view $\widetilde{M}= L(\Lambda)\bar\otimes L(\Gamma_1)*_{L(\Gamma)\bar\otimes L(\Sigma)}L(\Lambda)\bar\otimes L(\Gamma_2) $.  Let $A\subset A_1$ be an arbitrary diffuse amenable subalgebra.  Naturally, $\Delta^\pi(A)\subset \Delta^\pi (A_1) $  is an amenable diffuse  subalgebra of $\widetilde{M} $.  Thus by
 \cite[Theorem A]{Va13} we have the following trichotomy:
\begin{enumerate}[label=(\alph*)]
\item \label{g1}$\Delta^\pi(A)\prec_{\widetilde{M}} L(\Lambda)\bar\otimes L(\Sigma) $, or
\item \label{g2}$\Delta^\pi(A_2)\prec_{\widetilde{M}} L(\Lambda)\bar\otimes L(\Gamma_i) $ for some $i=1,2 $, or
\item \label{g3}$\Delta^\pi(A) $ is coamenable relative to $L(\Lambda)\bar\otimes L(\Sigma) $ inside $\widetilde{M} $. 
\end{enumerate}
Assume \ref{g3} holds.  Then  a repeated application of \cite[Theorem A]{Va13} gives either
\begin{enumerate}[label=(\alph*),resume]
\item \label{g4}$\Delta^\pi(A_2)\prec_{\widetilde{M}} L(\Lambda)\bar\otimes L(\Sigma) $, or
\item \label{g5}$\Delta^\pi(A_1\vee A_2)\prec L(\Lambda)\bar\otimes L(\Gamma_i) $ for some $i=1,2 $, or
\item \label{g6}$\Delta^\pi(A_1\vee A_2) $ is coamenable relative to $ L(\Lambda)\bar\otimes L(\Sigma) $ inside $\widetilde{M} $.
\end{enumerate}
We claim \ref{g3} will necessarily imply \ref{g4} as cases \ref{g5} and \ref{g6} will necessarily yield a contradiction. 
If \ref{g5} were to hold, this may further implies $\Delta^\pi (L(\Lambda))\prec_{\tilde{M}}L(\Lambda)
\bar\otimes L(\Gamma_i)  $ since $[M:A_1\vee A_2]<\infty $ and hence by \cite[Proposition 8.5]{CK15} we have $L(\Lambda)\prec_M L(\pi^{-1}(\Gamma_i)) $.  This gives $[\Lambda: \pi^{-1}(\Gamma_i)]<\infty $ which in turn implies $[\Gamma: \Gamma_i]<\infty $ which can hold only if $\Gamma_i=\Sigma $, a contradiction.\\
If \ref{g6} holds, since $[M:A_1\vee A_2]<\infty $ this would further imply $\Delta^\pi(L(\Gamma)) $ coamenable relative to $L(\Lambda)\bar\otimes L(\Sigma) $ inside $\widetilde{M} $  by \cite{OP07}.  Using \cite[Proposition 8.6]{CK15}, then $\Gamma=\pi(\Lambda)  $ is amenable relative to $\Sigma $ inside $\Gamma $.  However, \cite{Io12} would imply $[\Gamma_i,\Sigma]\leq 2 $ for $i=1,2$, again leading to a contradiction.  Hence if $\Delta^\pi(A) $  is amenable relative to $L(\Lambda)\bar\otimes L(\Sigma) $ inside $\widetilde M $, then $\Delta^\pi(A_2)\prec_{\widetilde M} L(\Lambda)\bar\otimes L(\Sigma) $. 

Now if we suppose \ref{g2} holds, we claim this also implies $\Delta^\pi(A_2)\prec_{\widetilde{M}} L(\Lambda)\bar\otimes L(\Sigma) $.  If $\Delta^\pi(A_2)\prec_{\widetilde{M}}L(\Lambda)\bar\otimes L(\Gamma_i) $,  then there exists projections $ p\in \Delta^\pi(A_2)$, $q\in L(\Lambda)\bar\otimes L(\Gamma_i)  $, a partial isometry $v\in \widetilde{M} $,   and a $* $-isomorphism $\Psi:p \Delta^\pi(A_2)p\to B\subset q   L(\Lambda)\bar\otimes L(\Gamma_i)q$ such that
\begin{equation}\label{f1}
\Psi(x)v=vx \quad \forall x\in p \Delta^\pi(A_2)p.
\end{equation} 
Additionally, we may assume $\Delta E_{L(\Lambda)\bar\otimes L(\Sigma)}(v^*v)=q $.  
Next, if  $B\prec_{L(\Lambda)\otimes L(\Gamma_i)} L(\Lambda)\otimes L(\Sigma)$, there exists $\Phi:rBr\rightarrow tL(\Lambda)\bar\otimes L(\Sigma)t$ such that
\begin{equation}\label{f2}
\Phi(x)w=wx \, \text{for all} \, x\in rBr,\, \text{where}\,  r\leq q \,\text{ and } w\in L(\Lambda)\bar\otimes L(\Gamma_i)
\end{equation}
Using (\ref{f1}) together with (\ref{f2}), we get that
\begin{equation}\label{f3}
\Phi(\Psi(x))wv=w\Psi(x)v=wvx.
\end{equation}
Note that if $wv=0$, we have $wvv^*=0$, $wE_{L(\Lambda)\bar\otimes L(\Gamma_i)}(vv^*)=0$. Then we obtain
$wq=w\Delta E_{L(\Lambda)\bar\otimes L(\Gamma_i)}(vv^*)=0$. Therefore, $w=0$, a contradiction.
As the consequence, $wv\not= 0$ and hence (\ref{f3}) show that 
\begin{equation}\label{f4}
\Delta^{\pi}(A_2) \prec L(\Lambda)\bar\otimes L(\Sigma).
\end{equation}
If $B\not\prec L(\Lambda)\bar\otimes L(\Sigma)$, by \cite[Theorem 1.2.1]{IPP05} we get $vv^*\in L(\Lambda)\bar\otimes L(\Gamma_i)$ and hence $Bvv^*=v\Delta^{\pi}(A_2)v^*\subset L(\Lambda)\bar\otimes L(\Gamma_i)$. 
Since $Bvv^*\not\prec L(\Lambda)\bar\otimes L(\Sigma)$, again by \cite[Theorem 1.2.1]{IPP05} it follows that $q\mathcal{N}_{L(\Lambda)\otimes L(\Gamma_i)}(v\Delta^{\pi}(A_2)v^*)''\subset L(\Lambda)\bar\otimes L(\Gamma_i)$ and hence $v\Delta^{\pi}(A_1\vee A_2)v^*\subset L(\Lambda)\bar\otimes L(\Gamma)$ and $\Delta^{\pi}(A_1\vee A_2)\prec L(\Lambda)\bar\otimes L(\Gamma_i)$. Since $[M:A_1\vee A_2]<\infty$, we have $\Delta^{\pi}(L(\Lambda))\prec L(\Lambda)\bar\otimes L(\Gamma_i)$ which as before leads to a contradiction. Thus, case \ref{g2} always leads to (\ref{f4}), i.e., $\Delta^{\pi}(A_2)\prec L(\Lambda)\bar\otimes L(\Sigma)$.

Summarizing the results above we have for every amenable diffuse subalgebra $A\subset A_1 $ either
\begin{center}
$\Delta^\pi(A)\prec_{\widetilde{M}}L(\Lambda)\bar\otimes L(\Sigma) $\quad or \quad
$\Delta^\pi(A_2)\prec_{\widetilde{M}}L(\Lambda)\bar\otimes L(\Sigma)$.
\end{center}
By \cite[Corollary F.14, Appendix F]{BO08}, $\Delta^\pi(A_i)\prec_{\widetilde{M}} L(\Lambda)\bar\otimes L(\Sigma) $ for some $i\in \set{1,2} $.  A final application of \cite[Proposition 8.5]{CK15} allows us to conclude $A_i\prec_M L(\pi^{-1}(\Sigma))$, as desired.
\end{proof}

\begin{defn} A subgroup $\Sigma <\G$ is called \emph{weakly malnormal} if there exists $\g_1,...,\g_n\in \G$ satisfying $|\cap^n_{i=1} \g_i \Sigma \g_i^{-1}|<\infty$.  
\end{defn}

\begin{cor}\label{intertkernel} Assume Notation \ref{1} above and suppose that either  $\G=\G_1\ast_\Sigma\G_2$ or $\G= {\rm HNN}(\La,\Sg,\theta)$ is an icc group satisfying the conditions from Corollary \ref{intertwiningcoregroupsafphnn}. Assume also that $\Sigma <\G$ is a weakly malnormal subgroup. Let $p\in L(\La)$ be a projection and let $A_1,A_2\subset pL(\La)p$ be diffuse commuting subalgebras such that $A_1\vee A_2= pL(\La)p$. Then there exists $i=1,2$ such that $A_i\prec L(\ker(\pi))$. 
\end{cor}
\begin{proof} This follows directly by combining Theorem \ref{intertquotient} and  \cite[Corollary 7, Proposition 8]{HPV10}\end{proof}

\begin{thrm}[Theorem in \cite{CKP14}]\label{intertquotientCrss} Assume the Notation \ref{1} above and suppose that  $\G\in \mathcal C_{\rm rss}$. Let $p\in L(\La)$ be a projection and let $A_1,A_2\subset L(\La)$ be diffuse commuting subalgebras such that $A_1\vee A_2\subset pL(\La)p$ has finite index. Then there exists $i=1,2$ such that $A_i\prec L(\ker(\pi))$. 
\end{thrm}

\begin{thrm}\label{almosttensor} Let $\G\in Quot_n(\mathcal{T})\cup Quot_n(\mathcal{T})$-by-$\mathcal C_{\rm rss}$. Let $r\in L(\G)$ be a projection and assume $P_1,P_2\subseteq rL(\G)r$ are commuting diffuse subfactors such that $P_1\vee P_2\subseteq rL(\G)r$ is finite index. Then there exist icc commuting subgroups $\Sigma_1, \Sigma_2<\G$ such that  $[\G:\Sigma_1\Sigma_2]<\infty$ and  $P_1\cong^{com}_{L(\G)} L(\Sg_1)$, $P_2\cong^{com}_{L(\G)} L(\Sg_2)$. 
 \end{thrm}

\begin{proof} For ease of notation denote by $M=L(\Gamma)$. We will proceed by induction on $n$. 

Assume $n=1$. In this case $\G$ either belongs to $\mathcal{T}$ or it is a $\mathcal{T}$-by-$\mathcal C_{\rm rss}$ group. 

First assume $\G=\G_1\ast_\Sg \G_2$.  Since $P_1\vee P_2\subseteq pMp$ is finite index, by Corollary \ref{intertwiningcoregroupsafphnn} we can assume $P_1 \prec L(\Sigma)$. Since any corner $eL(\Sigma)e$ is virtually prime then using Lemma \ref{intertwiningdichotomy} we must have that  $P_1\cong_M^{com} L(\Sigma)$ and the  desired conclusion follows by applying Theorem \ref{virtualprod}. The case  $\G={\rm HNN}(\La,\Sg,\theta)$ follows similarly.
 
Now assume $\G$ is a $\mathcal{T}$-by-$\mathcal C_{\rm rss}$ group. Let $\pi: \G \ra \La$ be a surjective homomorphism so that $\La\in \mathcal C_{\rm rss}$ and  $\ker(\pi)\in \mathcal{T}$. By Theorem \ref{intertquotientCrss} we have $P_1\prec L(\ker(\pi))$. Using Lemma \ref{intertwiningdichotomy} one of the following cases must hold:

\begin{enumerate} 
\item\label{p1'}  There exist projections $p\in P_1, e\in L(\ker(\pi)) $, a partial isometry $w\in M$,  and a unital injective $*$-homomorphism $\Phi: pP_1p\ra eL(\ker(\pi))e$ such that \begin{enumerate}
\item\label{21''} $\Phi(x)w=wx \text{ for all } x\in pP_1p$;
\item\label{22''} $s(E_{L(\ker(\pi))}(ww^*))=e$;
\item \label{23''} If $B:=\Phi(pP_1p)$ then $B\vee (B'\cap eL(\ker(\pi))e) \subseteq eL(\ker(\pi))e$ is a finite index inclusion of II$_1$ factors.
 \end{enumerate}
\item\label{p2'} $P_1\cong^{com}_M L(\ker(\pi))$.
\end{enumerate}

If case (\ref{p2'}) were to hold the the desired conclusion follows from Theorem \ref{virtualprod}.

Assume case (\ref{p1'}) holds. Since $\ker(\pi)\in \mathcal{T}$  then using (\ref{23''}) and the previous case one can find commuting subgroups $\Omega_1,\Omega_2 \leqslant \ker(\pi)$ such that $[\ker(\pi):\Omega_1\Omega_2]<\infty$ and $B\cong^{com}_{L(\ker(\pi)} L(\Omega_1)$ and $B'\cap eL(\ker(\pi))e\cong^{com}_{L(\ker(\pi))} L(\Omega_2)$ Proposition \ref{transitivity} we get that $P_1\cong^{com}_M L(\Omega_1)$ and the conclusion follows again from Theorem \ref{virtualprod}.

For the inductive step we argue in a similar manner. First assume $\G \in Quot_n(\mathcal{T})$ and suppose there is a surjective group homomorphism $\pi: \G \ra \La$  where  $\La$ is either an amalgamated free product of an HNN-extension and  $\ker(\pi)\in Quot_{n-1}(\mathcal{T})$. By Theorem \ref{intertquotient} we have $P_1\prec L(\pi^{-1}(\Sigma))$. 
 Using Lemma \ref{intertwiningdichotomy} one of the following cases must hold:

\begin{enumerate} 
\item\label{p1''}  There exist projections $p\in P_1, e\in L(\pi^{-1}(\Sigma)) $, a partial isometry $w\in M$,  and a unital injective $*$-homomorphism $\Phi: pP_1p\ra eL(\pi^{-1}(\Sigma))e$ such that \begin{enumerate}
\item\label{21'''} $\Phi(x)w=wx \text{ for all } x\in pP_1p$;
\item\label{22'''} $s(E_{L(\pi^{-1}(\Sigma))}(ww^*))=e$;
\item \label{23'''} If $B:=\Phi(pP_1p)$ then $B\vee (B'\cap eL(\pi^{-1}(\Sigma))e) \subseteq eL(\pi^{-1}(\Sigma))e$ is a finite index inclusion of II$_1$ factors.
 \end{enumerate}
\item\label{p2''} $P_1\cong^{com}_M L(\pi^{-1}(\Sigma))$.
\end{enumerate}
If case (\ref{p2''}) were to hold the the desired conclusion follows from Theorem \ref{virtualprod}.

Assume case (\ref{p1''}) holds.  Since $\pi^{-1}(\Sigma)$ is $Quot_{n-1}(\mathcal{T})$-by-$\mathcal C_{\rm rss}$ then using (\ref{23'''}) and the  induction hypothesis one can find commuting subgroups $\Omega_1,\Omega_2 \leqslant \pi^{-1}(\Sigma)$ such that $[\pi^{-1}(\Sg):\Omega_1\Omega_2]<\infty$ and $B\cong^{com}_{L(\pi^{-1}(\Sigma))} L(\Omega_1)$ and $B'\cap eL(\pi^{-1}(\Sigma))e\cong^{com}_{L(\pi^{-1}(\Sigma))} L(\Omega_2)$ Proposition \ref{transitivity} we get that $P_1\cong^{com}_M L(\Omega_1)$ and the conclusion follows again from Theorem \ref{virtualprod}.

When $\G\in Quot_n(\mathcal{T})$-by-$\mathcal C_{\rm rss}$ one can proceed in a similar manner and we leave the details to the reader.\end{proof}

\begin{thrm}\label{tensordecomppolyamalgams}  Denote by $\mathcal D= \mathcal C_{\rm rss}\cup (\text{finite-by-}\mathcal{T})^c$. Let $\G\in Quot_n(\mathcal{T})\cup (Quot_n(\mathcal{T})$-by-$\mathcal C_{\rm rss})$ be an icc group. Let $M=L(\G)$ and assume that $ M=M_1\bar\otimes M_2$, where $M_i$ are diffuse factors. Then there exist a product groups decomposition $\G=\G_1\times \G_2$, where $\G_i \in Quot^c_{n_i}( \mathcal D)$ for some $n_i\in \overline{1,n}$. Moreover, one can find a unitary $u\in M$ and $t>0$ such that $M_1=uL(\G_1)^tu^*$ and $M_2=uL(\G_2)^{1/t}u^*$.  
 \end{thrm}
\begin{proof} The conclusion follows directly by combining Theorems \ref{almosttensor}, \ref{fromrelcomtocomgroups} and \ref{prodstructureextensionamalagam}.\end{proof}

\subsection{Graph products of groups}  Let $\mathcal G$ be a graph without loops and multiple edges and denote its vertex set by $\mathcal V(\mathcal G)=\mathcal V$. Let $\{\G_v\}_{v\in \mathcal V}$ be a family of groups, indexed over the vertices of $\mathcal G$ and called vertex groups. The graph product of  $\{\G_v\}_{v\in \mathcal V}$, denoted by ${\mathcal G}(\G_v, v\in \mathcal V)$, is defined as the quotient of the free product $\ast_{v\in \mathcal V} \G_v$ by the relations $[\G_u,\G_v]=1$ whenever $u$ and $v$ are adjacent vertices in $\mathcal G$. These groups were introduced by E. Green in \cite{Gr90} and are natural generalizations of the right angled Artin and Coxeter groups. The study of graph products and their subgroups became a trendy subject over the last decade in geometric group theory and many important results have been discovered, \cite{A13,HW08,W11,MO13,AM10}.

Given any subset $U \subseteq \mathcal V$, the subgroup $\G_{ U}=\langle\G_u:u\in U\rangle$ of ${\mathcal G}( \G_v, v\in \mathcal V)$ is called a \emph{full subgroup}. In turn this can be naturally identified with the graph product $\mathcal G_{ U}(\G_u, u\in U)$ corresponding to the subgraph $\mathcal G_U$ of $\mathcal G$, spanned by the vertices of $U$. Any conjugate of a full subgroup is called a \emph{parabolic subgroup}. 

For every $v\in \mathcal V$ we denote by ${\rm link}(v)$ the subset of vertices $w\neq v$ that are adjacent to $v$ in $\mathcal G$. Similarly, for every $U\subseteq \mathcal V$ we denote by ${\rm link}( U)=\cap_{u\in U }{\rm link}(u) $. Also we make the convention that ${\rm link}(\emptyset)= \mathcal V$.  Notice that $ U\cap {\rm link}( U)=\emptyset$.

Graph product groups naturally admit many amalgamated free product decompositions. For further use we recall from \cite[Lemma 3.20]{Gr90} the following free amalgam decomposition involving parabolic subgroups associated with link sets of $\mathcal G$; precisely, for any $v\in \mathcal V$ we have  
\begin{align}\label{111}\mathcal{G}(\Gamma_w, w\in \mathcal{V}) = \G_{\mathcal V\setminus \{v\}}\ast_{\G_{\rm link}(v)} \G_{{\rm star}(v)},
\end{align}    
where we denoted by ${\rm star}(v)=\{v\}\cup {\rm link}(v)$.

 Let $\mathcal G$ be a graph and let $\mathcal V$ be its vertex set. We denote by $\mathcal V_{\rm star}$ the subset of all $v\in \mathcal V$ such that ${\rm star}(v)=\mathcal V$. With this notation at hand we recall the following well-known result which will be useful in the proof of the main result. 
 
 \begin{prop}\label{prop:ProductGraphProd} Assume $\mathcal{G} $ is a graph let $\Gamma=\mathcal{G}(\Gamma_v,v\in \mathcal{V}) $ be a graph product group, where $\mathcal V_{\rm star}=\emptyset$. If $\G=\La_1\times \La_2$ is a nontrivial product decomposition there exists $B\subset \mathcal V$ such that $\La_1=\G_B$ and $\La_2=\G_{{\rm link}(B)}$.
 \end{prop}

Also for further use we record the following basic result that follows from the study of parabolic subgroups of graph products performed in \cite[Section 3]{AM10} (see also \cite[Lemma 6.1]{MO13})

\begin{lem}\label{lem:FiniteIndexGraphGroup}
Assume $\mathcal{G} $ is a graph let $\Gamma=\mathcal{G}(\Gamma_v,v\in \mathcal{V}) $ be a graph product group, where $\Gamma_v $ is infinite for every $v\in \mathcal{V} $.  Assume there exist subsets $ U,  W \subseteq  \mathcal V$ and $\g\in \G$ satisfying  $\gamma \Gamma_{ W}\gamma^{-1}\leqslant \Gamma_{ U}$ and $[\Gamma_{ U}: \gamma \Gamma_{ W}\gamma^{-1}]<\infty $. Then $ U= W$ and $\Gamma_{ U}=\gamma\Gamma_{ W}\gamma^{-1}  $.  
\end{lem}

\begin{proof}
First, notice \cite[Corollary 3.8]{AM10} implies that $ W\subseteq  U $.  Since $[\Gamma_{ U}: \gamma \Gamma_{ W}\gamma^{-1}]<\infty $ there exist $\gamma_1,...,\gamma_k\in \Gamma_{ U} $ so that
$\bigcap_{i=1}^k \gamma_i\gamma \Gamma_{ W}{\gamma}^{-1}\gamma_i^{-1}$
is a finite index, normal subgroup of $\Gamma_{U} $.  Using \cite[Proposition 3.4]{AM10} inductively, one can find $\lambda\in \Gamma $ and $ T\subseteq  W$ such that $\bigcap_{i=1}^k \gamma_i\gamma \Gamma_{ W}{\gamma}^{-1}\gamma_i^{-1} =\lambda\Gamma_{ T} \lambda^{-1}$. Altogether, we have \begin{align}\label{100}\lambda \Gamma_{ T}\lambda^{-1}\leqslant \Gamma_{ U}\leqslant \mathcal{N}_\Gamma(\lambda\Gamma_{ T}\lambda^{-1}) \text{ and } [\Gamma_{\mathcal U}:\lambda \Gamma_{ T}\lambda^{-1}]<\infty.
\end{align} However, the normalizer formula of \cite[Proposition 3.13]{AM10} shows that 
\begin{align}\label{101}
\mathcal{N}_\Gamma(\lambda \Gamma_{ T}\lambda^{-1})=&\lambda \Gamma_{ T\cup{\rm link}( T)}\lambda^{-1}
= \lambda (\Gamma_{ T}\vee \Gamma_{{\rm link}( T)})\lambda^{-1}
= \lambda \Gamma_{ T}\lambda^{-1} \vee \lambda \Gamma_{{\rm link} ( T) }\lambda^{-1}.
\end{align}
Since $ T\cap {\rm link}( T)=\emptyset $ then $\lambda \Gamma_{ T}\lambda^{-1}, \lambda \Gamma_{{\rm link} (T) }\lambda^{-1}< \lambda \Gamma_{ T\cup{\rm link}( T)}\lambda^{-1}$ are commuting subgroups with trivial intersection. Then using (\ref{100}) and (\ref{101}) we get that $\G_{ U}= \lambda \Gamma_{ T}\lambda^{-1} \vee (\lambda \Gamma_{{\rm link} ( T) }\lambda^{-1}\cap \G_{\mathcal U})$. Since $[\Gamma_{ U}:\lambda \Gamma_{T}\lambda^{-1}]<\infty$ it follows that $\lambda \Gamma_{{\rm link} (T) }\lambda^{-1}\cap \G_{ U}$ is finite. By \cite[Proposition 3.4]{AM10} again there are $ X \subseteq {\rm link}( T)\cap  U$ and $\mu\in \G$ so that $\lambda \Gamma_{{\rm link} ( T) }\lambda^{-1}\cap \G_{ U}=\mu \G_{ X} \mu^{-1}$. However, since all non-trivial vertex groups are infinite it follows that $\G_{ X}=1$ and hence $\G_{ U}= \lambda \G_{ T}\lambda^{-1}$. Since, by construction, we have $ T \subseteq  W\subseteq  U$, this further entails $ T= W =  U$ and moreover $[\Gamma_U: \gamma\Gamma_W\gamma^{-1}]$ in the formula before  ; in particular,  $\g \G_{W} \g^{-1}=\G_{ U}$. 
\end{proof}

\begin{lem}
Let $\G=\mathcal{G}(\G_v, v\in \mathcal V)$ be an icc graph product group where $|\G_v|=\infty$ for all $v\in \mathcal V$.
If $M_1\bar\otimes M_2=L(\G)$, for diffuse $M_i$'s then there exists a partition $ A_1\sqcup A_2= \mathcal V\setminus \mathcal V_{\rm star}$ satisfying ${\rm link} (A_1)= A_2 \sqcup \mathcal V_{\rm star}$, ${\rm link} (A_2)= A_1 \sqcup \mathcal V_{\rm star}$, and 
\begin{align*}
M_i \prec^s L(\G_{{\rm link}(A_i)}), \text{ for all } i=1,2.
\end{align*}
\end{lem}

\begin{proof} If $\mathcal V= \mathcal V_{\rm star}$, the conclusion is trivial. So fix $v\in \mathcal V\setminus \mathcal V_{\rm star}$ arbitrary. Using the assumptions and formula (\ref{111}), there is a nontrivial  decomposition of $\Gamma $ as an amalgamated free product $\G=\G_{\mathcal V\setminus \{ v\} }*_{\G_{\operatorname{link}(v)}}\G_{\operatorname{star}(v)}$. Thus we have $M_1\bar\otimes M_2=L(\G_{\mathcal V\setminus\{v\}})*_{L(\G_{\operatorname{link}(v)})}L(\G_{\operatorname{star}(v)})$, and, using Corollary \ref{intertwiningcoregroupsafphnn}, we  must have either $M_1\prec L(\G_{\operatorname{link}(v)})$ or $M_2\prec L(\G_{\operatorname{link}(v)})$.
Since $\mathcal{N}_M(M_i)''=M$ is a ${\rm II}_1$ factor, \cite[Lemma 2.4(2)]{DHI16} further implies 
\begin{center}
$M_1\prec^s L(\G_{{\rm link}(v)})\quad$ or $\quad M_2\prec^s L(\G_{\operatorname{link}(v)})$.
\end{center}
For each $i=1,2$ denote by $ A_i\subset \mathcal V$ the maximal subset such that  
\begin{equation}\label{112}
M_i\prec^s L(\G_{\operatorname{link}(v)}) \text{ for all }\,v\in A_i.
\end{equation}

Notice that $A_1\cup A_2=\mathcal V\setminus \mathcal V_{\rm star}$.
Next we show the following 

\begin{claim}\label{116}For each $i=1,2$ we have $M_i \prec^s L(\G_{{\rm link}(A_i)})$.
\end{claim}
\vskip 0.03in
\noindent \emph{Proof of Claim 2.7.} Throughout this proof we will follow the notation used in \cite[Definition 2.2]{Va10}. Using (\ref{112}) and  \cite[Lemma 2.5]{Va10} we have 
\begin{equation*}
(M_i)_1\subset_{\rm approx} L(\G_{\operatorname{link}(v)}) \text{ for all } v\in A_i.
\end{equation*}
Thus, using \cite[Lemma 2.7]{Va10} successively, we obtain 
\begin{equation}\label{113}
(M_i)_1\subset_{\rm approx} L(\mathcal S)\text{, where }\mathcal S=\{\displaystyle\bigcap_{v\in A_i} h_v \G_{\operatorname{link}(v)}h_v^{-1} \,,\, h_v\in \G\}.
\end{equation}
Also,  using \cite[Proposition 3.4]{AM10} successively, for each $\{h_v\}_{v\in A_i}\subset \G$ there are $T\subset \bigcap_{v\in A_i}\operatorname{link}(v)=\operatorname{link}(A_i)$ and $k\in \G$ so that 
\begin{equation}\label{114}
\bigcap_{v\in A_i} h_v \G_{{\rm link}(v)} h_v^{-1} = k\G_T k^{-1} \leqslant k \G_{\operatorname{link}(A_i)} k^{-1}. 
\end{equation} Altogether, (\ref{113}), (\ref{114}) and \cite[Lemma 2.5]{Va10} show that $M_i\prec^s L(\G_{\operatorname{link}(A_i)})$, as desired. \hfill$\blacksquare$
\vskip 0.03in
Claim \ref{116} implies that $L(\G_{\operatorname{link}(A_i)})'\cap M\prec M_i'\cap M$.  Since $L(\G_{A_i})\subseteq L(\G_{\operatorname{link}(A_i)})'\cap M$ we further get $L(\G_{A_1})\prec M_2$. Since $M_2\prec^s L(\G_{\operatorname{link}(A_2)})$ then \cite[Lemma 3.7]{Va07}
 further leads to  
\begin{equation*}
L(\G_{A_1})\prec L(\G_{\operatorname{link}(A_2)}). 
\end{equation*}
Similarly we have $L(\G_{A_2})\prec L(\G_{\operatorname{link}(A_1)})$. Assume wlog that $A_1\neq\emptyset$. Then by \cite[Lemma 2.2]{CI17} there is $h\in\G$ such that $[\G_{A_1}:h\G_{\operatorname{link}(A_2)}h^{-1}\cap \G_{A_1}]<\infty$. Using \cite[Proposition 3.4]{AM10} again, there are $T\subset \operatorname{link}(A_2)\cap A_1$ and $s\in \G$ so that $h\G_{\operatorname{link}(A_2)} h^{-1}\cap \G_{A_1}=s\G_T s^{-1}$. Thus $s\G_Ts^{-1}\leqslant \G_{A_1}$ is finite index and by Lemma \ref{lem:FiniteIndexGraphGroup}  we have $A_1=T$ and also $s\G_Ts^{-1}= \G_{A_1}$. Since $T\subseteq \operatorname{link}(A_2)$ then $A_1\subseteq \operatorname{link}(A_2)$ and hence 
\begin{equation*} 
 \mathcal V=\mathcal V_{\rm star}\cup A_1\cup A_2\subseteq {\rm link}(A_2)\sqcup A_2\subseteq \mathcal V.
\end{equation*}
Therefore $A_2\sqcup {\rm link}(A_2)=\mathcal V$ and moreover $\mathcal V_{\rm star}\sqcup A_1={\rm link}(A_2)$ and $\mathcal V_{\rm star}\sqcup A_2={\rm link}(A_1)$. Thus Claim \ref{116} gives the desired conclusion.\end{proof}

\begin{thrm}\label{tensordecompgraphprod}
Let $\G=\mathcal{G}(\G_v, v\in \mathcal V)$ an graph product group. Assume that $\G$ is icc, $ \mathcal V_{\rm star}=\emptyset$, and $|\G_v|=\infty$ for all $v\in \mathcal V$.
If $M_1\bar\otimes M_2=L(\G)=M$, for diffuse $M_i$'s then there exists a proper subset $ A\subset \mathcal V$ such that $A\sqcup {\rm link}(A)=\mathcal V$, a unitary  $u\in M$, and a scalar $t>0$ such that 
\begin{align*}
M_1 = uL(\G_{{\rm link}(A)})^t u^*\quad\text{ and }\quad M_2 = uL(\G_A)^{1/t} u^*.
\end{align*}

\end{thrm}

\begin{proof} Using Lemma \ref{116} together with the hypothesis assumptions there exist proper subsets $A_1,A_2\subset \mathcal V$ so that $A_1 ={\rm link} (A_2)$, $A_2={\rm link}(A_1)$, and for each $i=1,2$ we have $M_i \prec^s L(\G_{{\rm link}(A_i)})$. So letting $A:=A_1$ we get 
\begin{align}\label{121}
M_1 \prec^s L(\G_{{\rm link}(A)})\quad\text{ and }\quad M_2\prec^s L(\G_A).
\end{align}
 Passing to the relative commutants intertwining in the above relations we also have $L(\G_A)\prec M_2$ and $
L(\G_{{\rm link}(A)})\prec M_1$. As $\mathcal N_{M}(L(\G_A))''=\mathcal N_{M}(L(\G_{\rm link (A)}))''=M$ is a II$_1$ factor we further conclude by \cite[Lemma 2.4(2)]{DHI16} that 
\begin{align}\label{122'}
 L(\G_A)\prec^s M_2\quad \text{ and }\quad
L(\G_{{\rm link}(A)})\prec^s M_1.
\end{align} Using (\ref{121})-(\ref{122'}) in combination with \cite[Theorem 6.1]{DHI16}, Lemma  \ref{lem:FiniteIndexGraphGroup}, and Proposition \ref{prop:ProductGraphProd} one can find $u\in \mathcal U(M)$ and $t>0$ such that $M_1 = uL(\G_{{\rm link}(A)})^t u^*$ and $M_2 = uL(\G_A)^{1/t} u^*$.
\end{proof}



\section{Applications to Prime Factors}\label{sec:ApplicationToPrimeFactors} Despite all the spectacular progress witnessed over the last decade, understanding the primeness aspects of AFP group factors $L(\G_1\ast_\Sigma \G_2)$ remains a nontrivial problem that seems to depend entirely on the ``nature'' and ``position'' of the amalgamated group $\Sigma$ inside $\G_1\ast_\Sigma\G_2$. While there is no satisfactory general answer to this problem we highlight below several natural conditions on $\Sigma$ that will insure primeness. This will lead to new examples of prime II$_1$ factors.

First we note the following particular case ($n=1$) of Corollary \ref{intertkernel}.

\begin{cor}\label{primeamalgam1} Let $\G=\G_1\ast_\Sigma \G_2$ be an amalgamated free product or $\G ={\rm HNN}(\La ,\Sigma, \theta)$ be an HNN extension as in the statement of Corollary \ref{intertwiningcoregroupsafphnn}. If $\G$ is icc and $\Sigma$ is weakly malnormal in $\G$ then $L(\G)$ is prime. 
\end{cor}

When $\Sigma$ is amenable we have the following strengthening of the previous result.

\begin{cor}\label{primeamalgam2} Let  $\G=\G_1\ast_\Sigma \G_2$ be an amalgamated free product or let $\G ={\rm HNN}(\La ,\Sigma, \theta)$ be an HNN extension as in the statement of Corollary \ref{intertwiningcoregroupsafphnn}. Also assume $\G$ is icc, $\Sigma$ is amenable, and there exists $\g\in \G$ such that $|\g\Sigma \g^{-1}\cap \Sigma| <\infty$. Then $L(\G)$ is virtually prime, i.e. for every commuting, diffuse subfactors $A_1,A_2\subset L(\G)$ we have that $[L(\G): A_1\vee A_2]=\infty$. 
\end{cor}

\begin{proof} Assume by contradiction that $M=L(\G)$ is not virtually prime, i.e. there exist diffuse, commuting subfactors $A_1,A_2 \subseteq M$ such that  $[M:A]<\infty$, where $A=A_1\vee A_2$. Applying Corollary \ref{intertwiningcoregroupsafphnn}  there exists $i\in\overline{1,2}$ so that  $A_i\prec L(\Sigma)$; assume $i=1$. Since $\Sigma$ is amenable, then $A_1$ is diffuse amenable. Thus $A$ is McDuff and hence $A'\cap A^\omega$ is a diffuse von Neumann algebra. By \cite[Proposition 1.11]{PP86} we have that $[M^\omega:A^\omega]=[M:A]<\infty$ and hence $[A'\cap A^\omega:A'\cap M^\omega]<\infty$. Altogether these show that $A'\cap M^\omega$ is diffuse. 

On the other hand since $[M:A]<\infty$ there exists $\{m_i\}\subset M$, a finite Pimsner-Popa basis. Using \cite[Lemma 3.1]{Po02} it follows that the map $\Phi(x)= [M:A]^{-1} \sum_i m_ixm_i^*$ implements the conditional expectation from $A' \cap M^{\omega}$ onto $M' \cap M^{\omega}$. In addition, the index of $\Phi$ is majorised by $[M:A]^{-1}$. Thus $[A' \cap M^{\omega}:M' \cap M^{\omega} ]<\infty$ and since $A'\cap M^\omega$ is diffuse we conclude that $M' \cap M^{\omega}$ is also diffuse; in particular $M$ has property Gamma and thus $\G$ is inner amenable. However, the assumptions and \cite[Corollary 2.2]{MO13} give that $\G$ must be acylindrically hyperbolic which contradicts \cite[Theorem 2.32]{DGO11} (see also \cite[Theorem C]{CSU13} for a more general argument).       
\end{proof}

The previous results lead to several new examples of groups that give rise to prime factors. In addition, we are also able to recover some of the primeness results proved in \cite[Corollary B]{DHI16}.

\begin{cor} If $\G$ is a group in one of the following classes then $L(\G)$ is prime.
\begin{enumerate}
\item [a)] the integral two-dimensional Cremona group $Aut_k (k[x,y])$, where $k$ is a countable field;
\item [b)] Higman's group $\langle a,b,c,d \, |\, a^b=a^2, b^c=b^2, c^d=c^2, d^a=d^2\rangle$ \cite{Hi51};
\item [c)] Burger-Mozes' groups \cite{BM01}, Camm's groups \cite{Ca51}, Bhattacharjee's groups \cite{Bh94};
\item [d)] $PSL_2(\mathbb Z[S^{-1}])$, where $S$ is a finite collection of primes;
\item [e)] $PGL_2(k[t])$, $PSL_2(k[t])$, where $k$ is a  countable field.
\item[f)] $PE_2(R[t]) =E_2(R[t])/(I_2)$, where $R $ is a finite commutative integral domain or a countable commutative integral domain with $|(r)|=\infty $ for every $r\in R\setminus\set{ 0} $.
\end{enumerate} 
\end{cor}

\begin{proof} a) It is well known that $\G=Aut_k (k[x,y])$ can be decomposed as an amalgamated free product $\G=A\ast_C B$ where:\begin{enumerate} \item $A$ is the group of all affine maps $\left ( \begin{array}{c} x \\ y \end{array}\right ) \ra \left ( \begin{array}{c} ax+by+c \\ dx+ey+f \end{array}\right ), $
where $a,b,....,f\in k$ and $\det\left(\begin{array}{cc}a,b\\c,d\end{array}\right)\neq 0$; 
\item $B$ is the triangular subgroup of all maps of the form $\left ( \begin{array}{c} x \\ y \end{array}\right ) \ra \left ( \begin{array}{c} ax+b \\ cy+g(x) \end{array}\right ), $  where $a,c\in k^*, b\in k$ and $g(x)\in k[x]$;
\item $C=A\cap B$ are all the triangular maps above where $g(x)$ is linear.
\end{enumerate}
Then using \cite[Lemma 4.23]{MO13} the conclusion follows from Corollary \ref{primeamalgam1}. In fact it is known that $C$ is nilpotent and hence amenable so the stronger conclusion from Corollary \ref{primeamalgam2} still holds. 

b) Higman group is an amalgam $\G=A\ast_C B$ where $A=\langle a,b,e \,|\, b^a=b^2 ,e^b=e^2\rangle$ and $B=\langle c,d,f \,|\, d^c=d^2,f^d=f^2\rangle $ with free amalgamated subgroups $C =\langle a,e\rangle$  and $\langle c,f\rangle $ with $a=f$, $e=c$. The result follows from Theorem \ref{tensordecompamalgam} and  \cite[Corollary 4.26]{MO13}.

c) Assume by contradiction $L(\G)=M_1\bar\otimes M_2$ for diffuse $M_i$'s. Note in all these cases $\G= A\ast_C B$ where $C<A$ and $C<B$ are finite index inclusions of non-abelian free groups. By \cite{Oz03} $L(C)$ is solid, hence virtually prime. Using Theorem \ref{tensordecompamalgam} we get $\G = C\times H$  where $H= A/C \ast B/C$, contradicting the simplicity of $\G$, \cite{BM01,Ca51,Bh94}.

d) First write $S=\{p_1,\ldots,p_n\}=T\cup\set{p_n} $.  By  \cite[Corollary 2 p.~80]{Se03},  $SL_2{Z[S]^{-1}} $ admits a decomposition as an amalgamated free product
$$SL_2(\mathbb Z[S^{-1}])=SL_2(\mathbb Z [T^{-1}])\ast_{\Gamma_0(p_n)} SL_2( \Z [T^{-1}]),$$ where $\Gamma_0(p_n)$ is the subgroup of all matrices whose lower-left entry is $0\mod p_n$, the Hecke congruence subgroup of level $p_n $ in $SL_2(\Z [T^{-1}])$. Now observe that $[SL_2(\Z[T^{-1}]:\Gamma_0(p_n)] =p_n+1$.
 Thus, after we mod out by the center $\pm I $, we obtain the presentation of $PSL_2(\mathbb Z[S^{-1}]) $ as the  amalgamated free product
\begin{align*}
PSL_2(\mathbb Z[S^{-1}])=PSL_2(\mathbb Z[T^{-1}])*_{\tilde\Gamma_0(p_n)}PSL_2(\mathbb Z[T^{-1}]),
\end{align*}
where $\widetilde\Gamma_0(p_n)\cong \Gamma_0(p_n)/(\pm I) $.  
Furthermore, $ PSL_2(\Z)$ contains a free subgroup of finite index, thereby showing 
 the group $\tilde\Gamma_0(p_n)$ is commensurable to a nonabelian free group and hence $L(\widetilde\Gamma_0(p_n))$ is virtually prime by \cite{Oz03}. 
 
If $L(PSL_2(\Z[S^{-1}] )) $  is non-prime,  Theorem \ref{tensordecompamalgam} implies that $PSL_2(\Z[S^{-1}] $ decomposes as a direct product of infinite groups, which is impossible since $PSL_2(\Z[S^{-1}]) $ is an irreducible lattice in $SL_2(\R)\times \prod_{p\in S} SL_2(\Q_p) $. 

e)   Letting $R$ be an arbitrary unital ring, we denote by $B(R) $ ( or $SB(R)$) as the subgroup of $GL_2(R) $ ( resp.~$SL_2(R) $) 
\begin{align*}
B(R):=\set{\left.
\begin{pmatrix}
a&b\\0& d 
\end{pmatrix}
\right\vert  a,d\in R^*, d\in R
},\\ SB(R):=\set{\left.
\begin{pmatrix}
a&b\\0& a^{-1} 
\end{pmatrix}
\right\vert  a\in R^*, d\in R}.
\end{align*} 
Nagao's Theorem (\cite[Theorem 6 p.~86]{Se03}) allows us to decompose $GL_2(k[t]) $ and $SL_2(k[t])$ as into the amalgamated free products $GL_2(k[t])=GL_2(k)*_{B(k)}B(k[t]) $ and
$SL_2(k[t])=SL_2(k)*_{SB(k)}SB_2(k[t]) $, respectively.  Taking the central quotient then yields the following decompositions:
\begin{align*}
PGL_2(k[t])=&PGL_2(k)*_{PB(k)}PB(k[t])\\
PSL_2(k[t])=&PSL_2(k)*_{PSB(k)}PSB(k[t])
\end{align*}
Where $PB(R):=B(R)/( R^*\cdot I_2) $ and $PSB(R)= SB(R)/( R^*\cdot I_2)	 $.

Now if $k $ were finite, $PB(k) $ and $PSB(k) $ are of type I and thus the result follows from \cite[Theorem 5.2]{CH08}.

Next  assume $k$ is a countably infinite field and note this is sufficient to conclude that $PB_2(k) $ and $PSB_2(k) $ are icc.  Now, there exists   $g\in SL_2(k)\leqslant GL_2(k[t])$ so that $B(k)^g $ and $SB(K)^g $ are a subgroups of the lower triangular matrices with entries in $k$.  
\begin{align*}
B(k)^g\cap B(k)=\set{\left.
\begin{pmatrix}
a&0\\ 0 & b
\end{pmatrix} \right\vert a,b\in k^*   }\cong k^*\times k^*, \\
SB(k)^g\cap SB(k)=\set{\left.
\begin{pmatrix}
a&0\\ 0 & a^{-1}
\end{pmatrix} \right\vert a\in k^*   }\cong k^*. 
\end{align*}
Now, if $P_1\bar\otimes P_2\cong L(PGL_2 (k[t])) $ is a decomposition into II$_1$ factors.  Corollary \ref{intertwiningcoregroupsafphnn}  allows us to assume  $P_1\prec L(PB(k))$, which moreover implies  $P_1\prec L(PB(k)\cap PB(k)^g) $ \cite[Theorem 6.16]{PV06}, c.f.~\cite[Corollary 7]{HPV10} .  However, this now implies that $ P_1$ is not a II$_1$ factor since $ PB(k)\cap PB(k)^g$ is abelian, a contradiction.
The case of $PSL_2(k[t]) $ follows analogously.

f) For any ring unital $S $, $E_2(S) $ denotes the subgroup of $SL_2(S) $ generated by the elementary matrices, and note that if $S $ is not a field, the inclusion is in fact proper (see \cite{KM97}).  By
Nagao's theorem, we write  $E_2(R[t])\cong E_2(R)*_{SB(R)} SB(R[t]) $ and we see the case when $R $ is finite follows again from an application of \cite[Theorem 5.2]{CH08}.    We remark that the  condition  $|(r)|=\infty $ for every $r\in R\setminus\set{0} $ is both necessary and sufficient to ensure that the group $PSB(R) $ is icc. Thus, after passing to the central quotient, we  proceed exactly as in part e). 
\end{proof}
\begin{Remarks}
We mention that one may relax the assumption that $k $ or $R$ are countable in parts e) and f).  Though the resulting von Neumann algebra is no longer seperable, i.e.~it does not act on a seperable Hilbert space, the techniques found in \cite{Po83} are adaptable to this situation and thus we conclude that $L(\Gamma) $ is prime if:
\begin{enumerate}
\item $\Gamma=PGL_2(k[t]) $ for every field $k$, or
\item $ \Gamma=PSL_2(k[t])$ for every field $k $, or
\item $\Gamma=PE_2(R[t]) $ for every finite integral domain $R$, or
\item $\Gamma= PE_2(S[t]) $ for every infinite integral domain such that $|(r)|=\infty $ for every $r\in R $.
\end{enumerate}
Finally, we close by mentioning our generalization of a conjecture of J.~Peterson which was our impetus for undertaking the current project.  
\end{Remarks}

\begin{conj} If $\G$ is a Burger-Mozes group \cite{BM01} or a Camm's group \cite{Ca51}, or a Bhattacharjee's group \cite{Bh94} then $L(\G)$ is  strongly solid. 
\end{conj}


\begin{thebibliography}{999}

\bibitem[A13]{A13} I. Agol, \textit{The virtual Henken conjecture}, Doc. Math. \textbf{18} (2013), 1045--1087. With an appendix by I. Agol, D. Groves, and J. Manning.

\bibitem[AM10]{AM10} Y. Antolín, A. Minasyan, \textit{Tits alternatives for graph products}, J. Reine Angew. Math. \textbf{704} (2015), 55--83.

\bibitem[Bh94]{Bh94} M. Bhattacharjee, \textit{Constructing finitely presented infinite nearly simple groups}, Comm. Algebra \textbf{22} (1994), 4561--4589. 

\bibitem[Bo12]{Bo12} R. Boutonnet, \textit{On solid ergodicity for Gaussian action}, J. Funct. Anal. \textbf{263} (2012), 1040--1063.

\bibitem[BHR12]{BHR12} R. Boutonnet, C. Houdayer, and S. Raum, \textit{Amalgamated free product type III factors with at most one
Cartan subalgebra}, Compos. Math. \textbf{150} (2014), 143--174.

\bibitem[BO08]{BO08} N. P. Brown, N. Ozawa, \textit{$\mathrm{C}^\ast$-algebras and finite-dimensional approximations}, Graduate Studies in Mathematics, vol. 88, AMS, Providence, RI.

\bibitem[BM01]{BM01} M. Burger, S. Mozes, \textit{Lattices in products of trees}, Inst. Hautes \`{E}tudes Sci.~Pub. S\'{e}r. I Math. \textbf{92} (2001), 151--194. 

\bibitem[Ca51]{Ca51} R. Camm, \textit{Simple free products}, J. London Math. Soc. \textbf{28} (1953), 66--76.

\bibitem[Ca16]{Ca16} M. Caspers, \textit{Absence of Cartan subalgebras for right-angled Hecke von Neumann algebras}, Preprint arXiv:1601.00593.

\bibitem[CF14]{CF14} M. Caspers, P. Fima, \textit{Graph products of operator algebras}, J. Noncommut. Geom. {\bf 11} (2017), 367--411.
 
\bibitem[CdSS15]{CdSS15}  I. Chifan, R. de Santiago, T. Sinclair, \textit{$W^*$-rigidity for the von Neumann algebras of products of hyperbolic groups}, Geom. Funct. Anal. \textbf{26} (2016), 136--159.

\bibitem[CH08]{CH08} I. Chifan, C. Houdayer, \textit{Bass-Serre rigidity results in von Neumann algebras}, Duke Math. J. \textbf{153} (2010), 23--54.

\bibitem[CI08]{CI08} I. Chifan, A. Ioana, \textit{Ergodic subequivalence relations induced by a Bernoulli action}, Geom. Funct. Anal. \textbf{20} (2010), 53--67.

\bibitem[CI17]{CI17} I. Chifan, A. Ioana, \textit{Amalgamated free product rigidity for group von Neumann algebras}, \textit{Preprint}. arXiv:1705.07350.


\bibitem[CIK13]{CIK13} I. Chifan, A. Ioana and Y. Kida, \textit{ $W^*$-superrigidity for arbitrary actions of central quotients of braid groups}, Math. Ann. \textbf{361} (2015), 563--582.

\bibitem[CK15]{CK15} I. Chifan, Y. Kida, \textit{OE and $W^*$-superrigidity results for actions by surface braid groups}, Proc. Lond. Math. Soc. \textbf{111} (2015), 1431--1470.

\bibitem[CKP14]{CKP14}I. Chifan, Y. Kida, S. Pant, \textit{Primeness Results for von Neumann Algebras Associated with Surface Braid Groups}, Int. Math. Res. Not. \textbf{16} (2016), 4807--4848.

\bibitem[CS11]{CS11} I. Chifan, T. Sinclair, \textit{On the structural theory of II$_1$factors of negatively curved groups}, Ann. Sci. \'{E}c. Norm. Sup. \textbf{46} (2013), no. 1, 1--34.

\bibitem[CSU11]{CSU11}  I. Chifan, T. Sinclair, and B. Udrea, \textit{On the structural theory of II$_1$ factors of negatively curved groups, II. Actions by product groups}, Adv. Math. \textbf{245} (2013), 208--236. 

\bibitem[CSU13]{CSU13} I. Chifan, T. Sinclair, B. Udrea, \textit{Inner amenability for groups and central sequences in factors}, Ergodic Theory Dynam. Systems \textbf{36} (2016), no. 4, 1106--1029.

\bibitem[Co76]{Co76} A. Connes, \textit{Classification of injective factor}, Ann.~of Math. \textbf{101} (1976), 73--115.

\bibitem[DI12]{DI12}  Y. Dabraowski and A. Ioana, \textit {Unbounded derivations, free dilations and indecomposability results for II$_1$ factors}, Trans. Amer. Math. Soc. \textbf{368} (2016), no. 7, 4525--4560. 

\bibitem[D69]{D69} J. Dixmier, \textit{Quelques propri\'et\'es des suites centrales dans les facteurs de type II$_1$}, Invent. Math. \textbf{7} (1969), 215--225.

\bibitem[DGO11]{DGO11} F. Dahmani, V. Guirardel, D. Osin, \textit{Hyperbolically embedded subgroups and rotating families in groups acting on hyperbolic spaces}, Mem. Amer. Math. Soc. \textbf{245} (2017), no. 1156, 1--164.


\bibitem[dC09]{dC09} Y. de Cornulier, \textit{Infinite conjugacy classes ingroups acting on trees}, Groups Geom. Dyn. {\bf 3} (2009) 267-277.

\bibitem[dSP17]{dSP17} R. de Santiago, S. Pant, \textit{Tensor decompositions for II$_1$ Factors of poly-hyperbolic groups}, Preprint, 2017.

\bibitem[DHI16]{DHI16} D.~Drimbe, D.~Hoff, A.~Ioana, \textit{Prime II$_1$ factors arising from irreducible lattices in products of rank one simple Lie groups}, J. Reine Angew. Math. to appear, arXiv:1611.02209.


\bibitem[Fi10]{Fi10} P. Fima, \textit{A note on the von Neumann algebra of a Baumslag-Solitar group}, C. R. Acad. Sci. Paris, Ser. I \textbf{349} (2011), 25--27.



 \bibitem[Ge96]{Ge96} L. Ge, \textit{On maximal injective subalgebras of factors}, Adv. Math. \textbf{118} (1996), no. 1,  34--70. 
 
\bibitem[Ge98]{Ge98} L. Ge, \textit{Applications of Free Entropy to Finite von Neumann Algebras, II}  Ann,. of Math. Second Series, \textbf{147}, No. 1 (Jan., 1998), pp. 143--157.

\bibitem[Gr90]{Gr90} E. Green, \textit{Graph Products of Groups} PhD Thesis, The University of Leeds, 1990, \texttt{http://etheses.whiterose.ac.uk/236/}.
 	
\bibitem[HW08]{HW08} F. Hagelund, D. Wise, \textit{Special cube complexes}, Geom. Funct. Anal. \textbf{17} (2008), 1551--1620.

\bibitem[Hi51]{Hi51} G. Higman, \textit{A finitely generated infinite simple group}, J. London Math. Soc. (2) \textbf{26} (1951), 61--64.

\bibitem[Ho15]{Ho15} D. Hoff, \textit{von Neumann Algebras of Equivalence Relations with Nontrivial One-Cohomology.} J. Funct. Anal. 270 (2016), no. 4, 1501--1536. 


\bibitem[HI15]{HI15} C. Houdayer, Y. Isono, \textit{Unique prime factorization and the bicentralizer problem for a class of type III factors}, Adv. Math. \textbf{305} (2017), 402--455.


\bibitem[HPV10]{HPV10} C. Houdayer, S. Popa, S. Vaes, \text{A class of groups for which every action is $W*$-superrigid}, Groups, Geometry, and Dynamics \textbf{7} (2013), 577--590.

\bibitem[HV12]{HV12} C. Houdayer,  S. Vaes. \textit{Type  III  factors  with  unique  Cartan  decomposition}, J. Math\'{e}matiques Pures et Appliqu\text'{e}es \textbf{100} (2013), 564--590.

\bibitem[Io12]{Io12} A. Ioana, \textit{Cartan subalgebras of amalgamated free product II$_1$ factors}, Ann. Sci. \'{E}c. Norm. Sup.(4) \textbf{48} (2015), no. 1, 71--130. 

\bibitem[IPP05]{IPP05} A. Ioana, J. Peterson, S. Popa, \textit{Amalgamated free products of w-rigid factors and calculation of their symmetry groups}. Acta Math. \textbf{200} (2008), 85--153. 

\bibitem[Is14]{Is14} Y. Isono, \textit{Some prime factorization results for free quantum group factors}, J. Reine Angew. Math. \textbf{722} (2017), 215--250.

\bibitem[Is16]{Is16} Y. Isono, \textit{On fundamental groups of tensor product II$_1$
factors}, preprint arXiv:1608.06426.
 

\bibitem[Jo98]{Jo98} P. Jolissaint, \textit{Central Sequences in the Factor Associated with the Thompson Group
F}, Annales de l'Institut Fourier \textbf{48} (1998), 1093--1106.

\bibitem[Jo81]{Jo81} V.F.R. Jones, \textit{Index for subfactors}, Invent. Math. \textbf{72} (1983), 1--25.

\bibitem[KS70]{KS70} A. Karrass, D. Solitar, \textit{The subgroups of a free product of two groups with an amalgamated subgroup}, Trans. Amer. Math. Soc. \textbf{150} (1970), 227--255.

\bibitem[KM97]{KM97} S. Krstic, J. McCool, \textit{Free quotients of $SL_2(R[x])$}, Proc. Amer. Math. Soc. 1\textbf{125} (1997), 1585--1588.

\bibitem[MO13]{MO13} A. Minasyan, D. Osin, \textit{Acylindrical hyperbolicity of groups acting on trees}, Math. Ann. \textbf{362} (2015), 1055--1105.


\bibitem[MvN36]{MvN36} F.J. Murray, J. von Neumann, \textit{On rings of operators}, Ann. Math. \textbf{37} (1936), 116--229.

\bibitem[MvN43]{MvN43} F.J. Murray, J. von Neumann, \textit{Rings of operators IV}, Ann. Math. {\bf 44} (1943), 716--808.

\bibitem[Oz03]{Oz03} N. Ozawa, \textit{Solid von Neumann algebras}, Acta Math. \textbf{192} (2004), 111--117.

\bibitem[Oz04]{Oz04} N. Ozawa, \textit{A Kurosh-type theorem for type II$_1$ factors}, Int. Math. Res. Not. (2006), Art. ID 97560, 21.

\bibitem[OP03]{OP03} N. Ozawa, S. Popa, \textit{Some prime factorization results for type II$_1$ factors}, Invent. Math. \textbf{156} (2004), 223--234. 

\bibitem [OP07]{OP07} N. Ozawa, S. Popa, \textit{On a class of II$_1$ factors with at most one Cartan subalgebra}, Ann. Math. \textbf{172} (2010), 713--749.


\bibitem[Pe06]{Pe06} J. Peterson, \textit{$L^2 $-rigidity in von Neumann algebras}, Invent. Math. \textbf{175} (2009) no. 2, 417--433.

\bibitem [PP86]{PP86} M. Pimsner, S. Popa, \textit{Entropy and index for subfactors},  Ann. Sci. \'{E}cole Norm. Sup. \textbf{19} (1986), 57--106.

\bibitem[Po83]{Po83} S. Popa, \textit{Orthogonal pairs of $*$-subalgebras in finite von Neumann algebras}, J. Operator Th. \textbf{9} (1983), no. 2, 253--268.


\bibitem[Po95]{Po95} S. Popa, \textit{Classification of subfactors and their endomorphisms}, CBMS Regional Conference Series in Mathematics, vol. 86, Published for the Conference Board of the Mathematical Sciences, Washington, DC; by the American Mathematical Society, Providence, RI, 1995.

\bibitem[Po02]{Po02} S. Popa, \textit{Universal construction of subfactors}, J. Reine Angew. Math. \textbf{543} (2002), 39--81. 

\bibitem[Po03]{Po03} S. Popa, \textit{Strong rigidity of II$_1$ factors arising from malleable actions of $w$-rigid groups I}, Invent. Math. \textbf{165} (2006), 369--408.


\bibitem[Po06]{Po06}S. Popa, \textit{On Ozawa's property for free group factors}, Int. Math. Res. Not. (2007), no. 11, 10pp.

\bibitem[PV06]{PV06} S. Popa, S. Vaes, \textit{Strong rigidity of generalized Bernoulli actions and computations of their symmetry groups.} Adv. in Math. \textbf{217} (2008), 833--872.

\bibitem[PV11]{PV11}S. Popa, S. Vaes, \textit{Unique Cartan decomposition for $\textrm{II}_1$ factors arising from arbitrary actions of free groups}, Acta Math. \textbf{212} (2014), 141--198.

\bibitem[PV12]{PV12} S. Popa, S. Vaes, \textit{Unique Cartan decomposition for $\textrm{II}_1$ factors arising from arbitrary actions of hyperbolic groups}, J. Reine Angew. Math. \textbf{694} (2014), 215--239. 


\bibitem[Sc71]{Sc71} P. Schupp, \textit{Small cancellation theory over free products with amalgamation}, Math. Ann. \textbf{193} (1971), 255--264.


\bibitem[Se03]{Se03} J-P. Serre, \textit{Trees}, Translated from the French original by John Stillwell. Corrected 2nd printing of the 1980 English translation. Springer Monographs in Mathematics. Springer-Verlag, Berlin, 2003. x+142 pp.

\bibitem[Si10]{Si10} T. Sinclair, \textit{Strong solidity of group factors from lattices in $SO(n,1)$ and $SU(n,1)$}, J. Funct. Anal. \textbf{260} (2011), no. 11, 3209--3221.

\bibitem[SW11]{SW11} J. O. Sizemore and A. Winchester, \textit{Unique prime decomposition results for factors coming from wreath products}, Pacific J. Math. \textbf{265} (2013), no. 1, 221--232.

\bibitem[Ue08]{Ue08} Y. Ueda, \textit{Remarks on HNN extensions in operator algebras}, Illinois J. Math. \textbf{52} (2008), 705--725.
 

\bibitem[Va07]{Va07} S. Vaes, \textit{Explicit computations of all finite index bimodules for a family of II$_1$ factors}, Annales Scientifiques de l'Ecole Normale Supérieure \textbf{41} (2008), 743--788.

\bibitem[Va10]{Va10} S. Vaes, \textit{One-cohomology  and  the  uniqueness  of  the  group  measure  space  decomposition  of  a II$_1$ factor}, Math. Ann. \textbf{355} (2013), no. 2, 661--696.

\bibitem[Va13]{Va13} S. Vaes, \textit{Normalizers inside amalgamated free products von Neumann algebras}, Publ. Res. Inst. Math. Sci. \textbf{50} (2014), 695--721.  



\bibitem[W11]{W11} D. T. Wise, \textit{Research announcement: the structure of groups with a quasiconvex hierarchy}, Electron. Res. Announc. Math. Sci. \textbf{16} (2009), 44--55.





\end{thebibliography}
\end{document}